\documentclass[smallextended,natbib,runningheads]{svjour3}
\journalname{}
\usepackage{natbib}
\usepackage{graphicx}
 \usepackage{epstopdf}
 \usepackage{pgf,tikz}
 \usepackage[all]{xy}
 \usepackage{times}
 \usepackage[LY1]{fontenc}

\usepackage{url}
\usepackage{amssymb,amsmath}
\usepackage{bbm}
\usepackage{color}
\usepackage{ulem}

 \newtheorem{thm}[theorem]{Theorem}
 \newtheorem{lem}[lemma]{Lemma}
\newtheorem{pro}[proposition]{Proposition}
 \newtheorem{cor}[corollary]{Corollary}
 
 \newtheorem{defn}[definition]{Definition}
 \newtheorem{rem}[remark]{Remark}
\DeclareMathOperator{\grad}{grad}
\DeclareMathOperator{\tr}{tr}

\DeclareMathOperator{\R}{\mathbb{R}}
\DeclareMathOperator{\N}{\mathbb{N}}

\DeclareMathOperator{\PP}{\mathbb {P}}

\newcommand{\us}{\underline{s}}

\newcommand{\transp}[1]{ {}^t #1}
\newcommand{\usigma}{\underline{\sigma}}

\begin{document}

\title{ Wishart exponential families \\ on cones related to $A_n$ graphs
}
\subtitle{}


\author{P. Graczyk         \and
        H. Ishi \and 
        S. Mamane 
}


\institute{P. Graczyk \at
              University of Angers \\
              \email{piotr.graczyk@univ-angers.fr}           
           \and
           H. Ishi \at
              Nagoya University \\ 
              \email{hideyuki@math.nagoya-u.ac.jp}
              \and 
              S. Mamane \at University of the Witwatersrand \\ \email{Salha.Mamane@wits.ac.za}
}

\date{Received: date / Revised: date}

\maketitle

\begin{abstract}
Let $G=A_n$ be the graph corresponding  to the  graphical model
of  nearest neighbour interaction in a Gaussian character. 
 We study Natural Exponential Families(NEF) of 
Wishart distributions on {convex} cones $Q_G$ and  $P_G$,
 where  $P_G$ is the cone
  of positive definite {real symmetric} matrices with obligatory zeros prescribed by $G$,
 {and $Q_G$ is the dual cone of $P_G$.}
   The Wishart NEF {that} we construct include
 Wishart distributions considered earlier by {\cite{lauritzen1996} and \cite{L-M}} for models
 based on decomposable graphs. 
Our approach is however  different and allows us to study the basic objects of Wishart NEF on the cones $Q_G$ and  $P_G$.
We determine Riesz measures generating  Wishart exponential families
on $Q_G$ and  $P_G${,} 
 and we give the quadratic construction of these Riesz measures and exponential families.
The mean, inverse-mean, covariance  and variance functions, as well as moments of higher order
   are studied and {their} explicit formulas are given.

\keywords{Wishart distribution \and graphical model  \and nearest neighbour interaction}
\end{abstract}

\section{Introduction}
\label{intro}

The classical Wishart distribution was first derived by \cite{wishart1928} as the distribution of the maximum likelihood estimator of the covariance
 matrix of the multivariate normal distribution. In the framework of graphical Gaussian models, the distribution of the maximum
likelihood estimator of $\pi(\Sigma)$,
{where $\pi$ denotes the canonical projection onto $Q_G$},
 was derived by \cite{dawid1993}, who called it the hyper Wishart distribution.
{\cite{dawid1993} also considered the hyper inverse Wishart distribution which is defined on $Q_G$ as the Diaconis-Ylvisaker conjugate prior distribution for $\pi(\Sigma)$,
and \cite{roverato2000} derived the so-called $G$-Wishart distribution on $P_G$, that is,
 the distribution of the  concentration matrix $K=\Sigma^{-1}$
 when $\pi(\Sigma)$ follows the hyper inverse Wishart distribution}.   
\cite{L-M} constructed two classes of multi-parameter Wishart distributions on the cones {$Q_G$ and $P_G$} associated to a decomposable graph
$G$  and called them type I and type II Wishart distributions, {respectively}. They are more flexible because they have multiple shape parameters.
{In fact, the type I and type II Wishart distributions 
generalize} the hyper Wishart distribution and {the G-Wishart} distribution respectively.

The Wishart exponential  families introduced and studied in this paper include
the  type I and type II Wishart distributions of Letac-Massam
on the cones {$Q_G$ and $P_G$} associated to $A_n$ graphs. Our methods, which are new and different
from methods of articles cited above,  simplify  in a significant way the Wishart theory
for graphical   models.

In \cite{graczykIshi} and in \cite{ishiHammamet} the theory of Wishart distributions on  general convex cones was developped, with a strong accent on the quadratic constructions and on applications to homogeneous cones . In this article we  apply for the first time the ideas and results of
 \cite{graczykIshi} to study important families of non-homogeneous cones.

 Applications in estimation and other practical aspects of Wishart {distributions} are intensely studied,
{cf. \cite{sugiura1988,tsukuma2006,konno2007,konno2009,kuriki2010}.}  
 
The focus of this work is   on non-homogeneous cones  $Q_{A_n}$ and $P_{A_n}$
appearing in the statistical theory of  graphical models,  corresponding to the practical model of nearest neighbour interactions.
In the Gaussian character $(X_1,X_2,\ldots,X_n)$,     non-neighbours $X_i,X_j$, $|i-j|>1$ are conditionally independent with respect to other variables.
This family of decomposable graphical models presents many advantages: it encompasses the univariate case ($A_1$), a complete graph ($A_2$), 
 a non-complete homogeneous graph ($A_3$) and an infinite number of non-homogeneous graphs
 ($A_n$, $n\geq 4$).  

The methods introduced in this article allowed to solve  in
\cite{GIMO}
 the Letac-Massam Conjecture on the cones $Q_{A_n}$. Together with the results of this article
 we achieve in this way the complete study of all classical objects of {an} exponential family
 for the Wishart  NEF on the cones $Q_{A_n}$.
 
Some of the results of our research may be extended to cones related to  all decomposable graphs
(work in progress).  Many  of them are however specific for the cones $Q_{A_n}$ and $P_{A_n}$
(indexation of Riesz and Wishart measures by $M=1,\ldots, n$,  Letac-Massam Conjecture,  Inverse Mean Map, Variance function).\\

{\it Plan of the article.} Sections  \ref{sec:1},  \ref{ReccConstr} and \ref{LaplaceSection}
provide the main tools  in order to define and to study the Wishart NEF on  the cones $Q_{A_n}$ and $P_{A_n}$.
 In Section \ref{sec:1}, useful notions of eliminating orders
 $\prec$ on $A_n$ and 
of generalized power functions $\delta^{\prec}_{\us}$ and $\Delta^{\prec}_{\us}$, $\us \in \R^n$ will be introduced on the cones  $Q_{A_n}$ and $P_{A_n}$ respectively. 
 In Theorem \ref{delta-Delta}, 
 a classical relation between the power functions   $\delta^{\prec}_{\us}$  and $\Delta^{\prec}_{-\us}$ is proved
as well as the dependence of $\delta^{\prec}_{\us}$ and $\Delta^{\prec}_{\us}$ on the maximal element $M$
of $\prec$ only. Thus, in the sequel of the paper,  only generalized power functions $\delta^{(M)}_{\us}$ and $\Delta^{(M)}_{\us}$  appear.
Next important tool of analysis of Wishart exponential families {are} 
recurrent construction of the cones $P_G$ and $Q_G$ and corresponding changes of variables.
{They} are introduced and studied in Section \ref{ReccConstr}, {and} are immediately applied in Section
\ref{LaplaceSection} in order to compute
the Laplace transform of generalized power functions $\delta^{(M)}_{\us}$ and $\Delta^{(M)}_{\us}
$
(Theorems \ref{laplacedelta} and \ref{laplaceDelta}).

In  Section  \ref{RieszWishartQ}, Wishart  natural exponential families on the cones $Q_{A_n}$ 
are defined, and all their classical objects {are explicitly} determined, beginning  with the
Riesz generating measures, Wishart densities, Laplace transform, mean and covariance.
In Theorem \ref{inverseMean} and Corollary \ref{invmeanexplicit}, an explicit formula for  the  inverse mean map 
is proved. It 
provides  an infinite number  of versions of  Lauritzen formulas for bijections between the cones
$Q_G$ and $P_G$. In Section \ref{VarianceSection} two explicit  formulas are given for the variance function of a Wishart family. The formula of Theorem \ref{ThNice} is surprisingly simple and similar
to the  case of the symmetric cone $S_n^+$.
Sections \ref{quadr_Q} and \ref{higher} are devoted to the quadratic constructions of Wishart exponential families  on $Q_G$ and to the computation of their  higher moments in Theorem
\ref{highMom}.  An interesting connection to the Missing Data statistics is {mentioned} and will be developped in a forthcoming paper.

Section \ref{WishartP} is on  Wishart  natural exponential families on the cones $P_{A_n}$
and follows a similar scheme as Section  \ref{RieszWishartQ}, however the inverse mean map and variance function are not available  on the cones  $P_{A_n}$. The analysis on these  cones  
is more difficult.

In  the last Section \ref{LM} we  establish the relations of the Wishart NEF 
defined and studied in our paper  with the  type I and type II
Wishart distributions from \cite{L-M}. Our methods give a simple proof of the formulas for Laplace transforms of type I and type II
Wishart distributions from   \cite{L-M}.

\section{Preliminaries on $A_n$ graphs and related cones}
\label{sec:1}

In this section we study properties of graphs $A_n$ that will be important  in the theory of Riesz measures and Wishart distributions
on the cones related to these graphs.
In particular, we characterize all the
eliminating  orders of vertices and we introduce  generalized power functions related  to such orders.
We show that they only depend on the maximal element $M\in\{1,\ldots,n\}$ of the order.

An undirected graph is a pair  $G=(V , \mathcal{E})$, where $V$ is a finite set and $\mathcal{E}$
is a subset of $\mathcal{P}_2(V)$, the set of all subsets of $\mathcal{E}$ with  cardinality two.
The elements of $V$ are called nodes or vertices and the elements of $\mathcal{E}$ are called edges.
If ${\{v, v'\}}\in \mathcal{E}$, then {$v$ and $v'$} are said to be adjacent and this is denoted by
{$v \sim v'$}. Graphs are visualized by representing each node by a point and each edge {$\{v, v'\}$} by a line with the nodes
{$v$ and $v'$} as endpoints.
{For convenience, we introduce a subset $E \subset V \times V$ defined by $E:= \{(v,v'): v \sim v'\} \cup \{(v,v): v \in V\}$.}


The graph  with  $V=\{v_1,v_2,\ldots,v_n\}$ and  
$\mathcal{E} = {\{ \{v_j, \,\, v_{j+1}\}: 1 \le j \le n-1 \}}$ is denoted {by} $A_n$ 
and represented as  $1-2-3-\hdots-n$.
An $n$-dimensional Gaussian model  $(X_v)_{v\in V}$ is said to be Markov with respect to a graph $G$ if
for any {$(v,\,\,v')\notin E$}, the random variables $X_{v}$ and $X_{v'}$ are conditionally independent given all the
other variables. The conditional independence relations encoded in $A_n$ graph are of the form:
  $X_{v_i} \perp X_{v_j}| (X_{v_k})_{k\neq i,j}$, for all $|i-j|>1$.
Thus, $A_n$ graphs correspond to nearest neighbour interaction models.
{In what follows, we often denote the vertex $v_i$ by $i$.}

Let $S_n$ be the space of {real} symmetric matrices of order $n$ and let $S_n^+\subset S_n$ be  the cone of positive definite matrices. 
The notation for a positive definite matrix $y$ is $y>0$.
{For a graph $G$,
let $Z_G \subset S_n$ be the vector space consisting of $y \in S_n$ } such that $y_{ij}=0$ if $(i\,,\,j)\notin E$.
Let $I_G=Z_G^*$ be the dual vector space with respect to the scalar product
${\langle y, \eta\rangle= \tr(y\eta)=\sum_{(i,j)\in E}y_{ij}\eta_{ij}, \ \ y \in Z_G,\,\eta\in I_G}.$
{In the statistical literature, the vector space $I_G$
	is commonly realised as the space of $n\times n$  symmetric
	matrices $\eta$, in which only the coefficients $\eta_{ij}$,
	$(i,j)\in E$, are given. We adapt this realisation of $I_G$ in this paper.
}

If $I\subset V$, we denote by $y_I$ the submatrix of $y\in Z_G$ obtained by extracting from $y$ the lines and the columns 
indexed by $I$. The same notation is used {for} {$\eta \in I_G$}. 
Let $P_{G}$ be the cone defined by
$P_{G} = \{y\in Z_G: y>0 \}$,
 and $Q_G \subset I_G$ the dual cone of $P_G$, that is,
$$Q_G=\{ \eta\in I_G:\ \forall y\in \overline{P_{G}}\backslash \{0\}\ \ \langle y, \eta \rangle > 0   \}.$$ 
A Gaussian vector model $(X_v)_{v\in V}$ is Markov with respect to $G$ if and only if {the concentration matrix $K = \Sigma^{-1}$ belongs to $P_G$.}

{
When $G = A_n$, the cone {$Q_{G}$} is described as
$Q_{G} = \{ \eta\in I_G: \eta_{\{i,i+1\}}>0, \,\, i=1, \dots, n-1 \}$.
Let $\pi={\pi_{I_G}}$ 
 be the projection of
$S_n$ onto $I_G$, $x \mapsto \eta$ such that $\eta_{ij}=x_{ij}$ if $(i,j) \in E$.
Then it is known (cf. \cite{L-M, andersson2010}) that  
 the mapping  $P_G \longrightarrow Q_G$, $y\longmapsto \pi(y^{-1})$ is a bijection.}

In the sequel, unless otherwise stated, $G=A_n$, 
\subsection{Eliminating Orders}
Different orders of vertices $v_1,v_2,\ldots,v_n$
should be considered in order to have a harmonious theory of Riesz  and Wishart distributions  on the cones related to $A_n$ graphs.
The orders that will be important in this work are called {\it eliminating orders of vertices} and will be presented now.
\begin{defn}
Consider a graph $G={(V,\mathcal{E})}$ and an ordering $\prec$ of the vertices of $G$. The 
set of future neighbours of a vertex $v$ is defined as 
 $ v^+ = \{ w \in V: v\prec w \; \textnormal{and} \; v\sim w\} .$
 The set of all predecessors of a vertex  $v\in V$ with respect to   $\prec$ is defined as
$ v^- = \{ u \in V: u\prec v\}$.

\end{defn}
	\begin{defn}\label{DefElimin}
 An ordering $\prec$ of the vertices of a graph  $G$ is  said to be an eliminating order 
 {if $v^+$ is complete for all $v\in V$}.
 \end{defn}

 In this section, we present  a characterization
 of  the   eliminating orders in the case of the graph $A_n$. An algorithm that generates all eliminating orders for a general 
 graph is given by \cite[]{chandran2003}.
 \begin{pro}\label{order}
 Consider a graph $A_{n}: 1-2-3-\hdots-n$. All  eliminating orders are
 obtained by an  intertwining  of two sequences
 $
 1\prec\hdots\prec M\ \ {\rm and}\ \ n\prec \hdots \prec M
 $
for an $M\in V$. There are $2^{n-1}$ eliminating orders on the graph $A_n$.
\end{pro}
\begin{proof}
Consider an eliminating order $\prec$ on $G=A_n$.  Since the minimal element of an eliminating order $\prec$ on a graph $A_n$
is one of the  exterior vertices $1,n$ of the graph, it starts with $1$ or $n$, say it is $1$.
It follows from Definition \ref{DefElimin} that an eliminating order without its minimal element forms again
an eliminating order on the graph $A_{n-1}$ obtained from $G$ by suppressing $1$ or $n$. The element following $1$
may be 2 or $n$. This recursive argument proves that   in an eliminating order the sequences $1\prec2\hdots\prec M$ and $n\prec n-1\prec \hdots \prec M$ must appear intertwined.
We also see that we construct in this way $2^{n-1}$ different orders.

Conversely, if  an order $\prec$ on $G$  is obtained by intertwining of  the sequences $1\prec2\hdots\prec M$ and $n\prec n-1 \prec \hdots \prec M$,
it follows that the sets $v^+$  of future neighbours of $v$ are singletons or empty(for $v=M$). Thus the order   $\prec$ is eliminating.
\end{proof}


\subsection{Generalized power functions}

In this section, we define and study  generalized power functions on the cones $P_G$ and $Q_G$.
{
Here we introduce useful notation.
For $1 \le i \le j \le n$, let $\{i : j \} \subset V$ be the set of $a \in V$ for which $i \le a \le j$.
Then, for $y \in Z_G$ and $1 \le i \le n$, the matrix $y_{\{1:i\}}$ is the upper left submatrix of $y$ of size $i$,
 and $y_{\{i : n\}}$ is the lower right submatrix of size $n-i+1$.  }
Recall that
on the cone $S_n^+$, the generalized power functions are 
$\Delta_{\us}(y) = \prod_{i=1}^n |y_{\{1:i\}}|^{s_i-s_{i+1}}$ and 
$\delta_{\us}(y) = \prod_{i=1}^n |y_{\{i:n\}}|^{s_{i}-s_{i-1}}$,  with $s_0=s_{n+1}=0$.

\begin{defn}\label{DEFdeltas}
For $\us \in \mathbb{R}^V$, setting $\det y^{}_{\emptyset} = 1=\det \eta^{}_{\emptyset }$,
 we define 
\begin{align}
\Delta_{\us}^{\prec}(y)
 &:= \prod_{v \in V}
 \Bigl( \frac{\det y^{}_{\{v\} \cup v^-}}{\det y^{}_{v^-}}
 \Bigr)^{s_v}
 \qquad (y \in P_G),
 \label{eqn:def_of_Delta}\\
\delta_{\us}^{\prec}(\eta)
 &:= \prod_{v \in V}
 \Bigl( \frac{\det \eta^{}_{\{v\} \cup  v^+}}{\det \eta^{}_{v^+}}
 \Bigr)^{s_v}
 \qquad (\eta \in Q_G).
\label{eqn:def_of_delta}
\end{align}
\end{defn}

Note that Definition \ref{DEFdeltas} applied to the complete graph with the usual order $1<\hdots <n$ gives $\Delta_{\us}$ and $\delta_{\us} $.
 For any $\us$ the following formula  $\delta_{\us}(y^{-1})=\Delta_{-\us}(y)$ is well known.
In Theorem \ref{delta-Delta} we find an analogous formula in the case of the cones $P_G$
and $Q_G$.

We will see in Theorem \ref{delta-Delta} that on the cones related to the graphs $A_n$,
different order-depending power functions $\Delta_{\us}^{\prec}$ and
$\delta_{\us}^{\prec}$  defined in Definition \ref{DEFdeltas} may be expressed in terms
of explicit "$M$-power functions" ${ \Delta}^{(M)}_{\us}$
and ${\mathcal \delta}^{(M)}_{\us}$ that will be defined {below}. They depend only on the choice of $M\in V$.
 
\begin{defn}\label{M-powers}
 Let $M\in V$, $y\in P_G$ and $\eta \in Q_G$.
 We define the $M$-power functions ${ \Delta}^{(M)}_{\us}(y)$ on $P_G$
 and  ${  \delta}^{(M)}_{\us}(x)$ on $Q_G$ by the following formulas:

\begin{equation}\label{P(M)}
 {\Delta}^{(M)}_{\us}(y)= \prod_{i=1}^{M-1} |y_{\{1:i\}}|^{s_i-s_{i+1}}|y|^{s_{M}}
 \prod_{i=M+1}^n |y_{\{i:n\}}|^{s_{i}-s_{i-1}},
\end{equation}

\begin{equation}\label{delta(M)}
 \delta_{\us}^{(M)}(\eta)= \frac{\prod_{i=1}^{M-1} |\eta_{\{i : i+1\}} |^{s_i}
 \prod_{i=M+1}^n |\eta_{\{i-1:i\}} |^{s_i}}
 {\prod_{i=2}^{M-1} \eta_{ii}^{s_{i-1}}\cdot \eta_{MM}^{s_{M-1} - s_M  + s_{M+1}}\cdot \prod_{i=M+1}^{n-1} \eta_{ii}^{s_{i+1}}}.
\end{equation}

\end{defn}
Observe that for $M=1,n$\, there are $n-1$ factors in the denominator of \eqref{delta(M)}, and for $M=2,\hdots n-1$ 
there are $n-2$ factors (powers of $\eta_{22}\hdots \eta_{n-1,n-1}$).

 The main result of this section is the following theorem.

 \begin{thm}\label{delta-Delta}
 Consider a graph $G=A_n$ with an  eliminating order $\prec$. Let $M$ be the maximal element with respect to $\prec$. 
 Then for all $y\in P_G$, {we have}
\begin{equation}\label{delta-Delta-P}
 \delta_{\us}^{\prec}(\pi(y^{-1})) =\Delta_{-\us}^{\prec}(y)=\Delta_{-\us}^{(M)}(y).
\end{equation}

\end{thm}
The proof of Theorem \ref{delta-Delta} is preceded by a series of elementary lemmas.

\begin{lem}\label{nonconnected}
 Let $y\in P_G$ and $i< j < j+1< k < m$.
  The determinant of the submatrix $y_{\{i \,: \,j\}\cup \{k \,: \,m\}}$ can be factorized as
$|y_{\{i \,: \,j\}\cup \{k \,: \,m\}}|=|y_{\{i \,: \,j\}}||y_{\{k\, : \,m\}}|$.
\end{lem}

\begin{lem}\label{determinant}
 Let $y\in P_G$ and  $\eta = \pi(y^{-1})$. Then for all $i, i+1\in V$, we have
\begin{equation*}
 \begin{vmatrix}
 \eta_{\{i,i+1\}}
 \end{vmatrix}= |y|^{-1}\vert y_{V\backslash\{i,i+1\}} \vert.
\end{equation*}

\end{lem}
\begin{proof}
 We repeatedly  use the cofactor formula for an inverse matrix. We use $\eta_{ii} = |y|^{-1} |y_{V\backslash \{i\}}|$ and show that
 $\eta_{i,i+1} =- y_{i,i+1}|y|^{-1} |y_{V\backslash \{i,i+1\}}|$.\, It follows that \\
 $\begin{vmatrix}
 \eta_{\{i,i+1\}}
 \end{vmatrix}= |y|^{-2} \vert y_{V\backslash\{i,i+1\}} \vert
\left[ |y_{\{i+1 \,: \,n\}}||y_{\{1 \,: \,i\}}|- y_{i,i+1}^2 |y_{\{1 \,: \,i-1\}}||y_{\{i+2 \,: \,n\}}| \right]$.
The last factor in brackets equals $|y|$.
\end{proof}

\begin{proof} (of Theorem \ref{delta-Delta})

\noindent
{\it Part 1:} $\delta_{\us}^{\prec}(\pi(y^{-1})) = {{\Delta}^{(M)}_{-\us}(y).}$
From Proposition \ref{order}, we have
$$
{i^+ = \begin{cases} \{i+1\} & \mbox{ if }i \le M-1, \\ \emptyset &\mbox{ if }i=M, \\ \{i-1\} & \mbox{ if }i \ge M+1. \end{cases}}
$$
{Using $\eta_{ii} = |y|^{-1} |y_{V\backslash \{i\}}|$ with $\eta=\pi(y^{-1})$ and Lemmas \ref{nonconnected} and \ref{determinant}},
 we get 
$\delta_{\us}^{\prec} (\pi(y^{-1})) =\Delta^{(M)}_{-\us}(y).$\\
\noindent
{\it Part 2:} $ {\Delta_{\us}^{\prec}(y)={\Delta}^{(M)}_{\us}(y) }$.
Let us first consider the eliminating order $\prec_M$ {given by}
\begin{equation}\label{order0}
1\prec_M 2\prec_M  \hdots\prec_M  M-1\prec_M  n\prec_M {n-1} \prec_M \hdots\prec_M  M+1\prec_M  M.
\end{equation}
Using $\eta_{ii} = |y|^{-1} |y_{V\backslash \{i\}}|$, {Lemmas \ref{nonconnected} and \ref{determinant} again}, 
 we get
 $\Delta_{\us}^{\prec_M} (y) = {\Delta}^{(M)}_{\us}(y).$

It is easy to see using Proposition \ref{order} and the factorization from Lemma \ref{nonconnected}
  that for any other eliminating order $\prec$, the factors of  $\Delta_{\us}^{\prec} (y)$
 under the powers $s_i$ are exactly the same as for $\prec_M$.
Indeed, if $i\leq M-1$, let $n-j$ be the largest vertex greater than {M} such that $n-j \prec i$.
Then, the factor under the power $s_i$ is
$$\frac{|y_{\{i\}\cup i^-}|}{|y_{i^-|}}= {\frac{|y_{\{1:i\}}||y_{\{n-j:n\}}|}{|y_{\{1:i-1\}}||y_{\{n-j:n\}}|}
= \frac{|y_{\{1:i\}}|}{|y_{\{1:i-1\}}|}}.$$
 A similar argument shows that this is also true for $i= M$ and for $i>M$.
\end{proof}

\begin{cor}\label{Cor_delta}
Let $\prec_1$ and $\prec_2$ be two eliminating orders on $G$ such that $\max_{\prec_1}V=\max_{\prec_2}V$.
Then \,$\delta_{\us}^{\prec_1} (\eta)=\delta_{\us}^{\prec_2} (\eta)$\, 
for all $\eta\in Q_G$.  If $\max_{\prec}V=M$ then we have $
\delta_{\us}^{\prec} (\eta)=\delta_{\us}^{(M)}(\eta).
$
\end{cor}

 
\section{ Recurrent construction of the cones $P_G$ and $Q_G$ and changes of variables}\label{ReccConstr}
In this section we introduce  very useful recurrent constructions of the cones $P_{A_n}$ and $Q_{A_n}$
from the cones $P_{A_{n-1}}$ and $Q_{A_{n-1}}$. There are two variants  of them for  $A_{n-1}:2-\hdots-n$ {and} $A_{n-1}:1-\hdots-{(n-1)}$.
Corresponding changes of variables for integration on $P_{A_n}$ and $Q_{A_n}$ are introduced.
\begin{pro} \label{changeOfVar}
 \begin{enumerate}
 \item For $n\geq 2$, let
$\Phi_n:\mathbb{R}^+\times \mathbb{R}\times P_{A_{n-1}}\longrightarrow P_{A_n}$,
$(a,b,z)\longmapsto y$ with \\
$${y}=A(b)
\begin{pmatrix}
 a & 0&\cdots & 0\\
 0\\
 \vdots & & {z}\\
 0
\end{pmatrix}
\transp{A}(b),
\qquad
A(b)=\begin{pmatrix}
 1 \\
 b & 1\\
 \vdots & & \ddots\\
 0&\ldots& 0&1
\end{pmatrix},
$$
{and let}
 $\Psi_n:\mathbb{R}^+\times \mathbb{R}\times Q_{A_{n-1}}\longrightarrow Q_{A_n}$,
$(\alpha,\beta,x)\longmapsto \eta$ with
$$\eta = \pi\left(
\transp{A}(\beta)
\begin{pmatrix}
 \alpha & 0&\cdots & 0\\
 0\\
 \vdots & & x\\
 0
\end{pmatrix}
A(\beta)\right).
$$
{Then the} maps $\Phi_n$ and $\Psi_n$  are bijections.
\item Let
$\tilde{\Phi}_n:\mathbb{R}^+\times \mathbb{R}\times P_{A_{n-1}}\longrightarrow P_{A_n}$,
$(a,b,z)\longmapsto \tilde{y}$ with 
$$ \tilde{y}=\transp{B}(b)
\begin{pmatrix}
  & & & 0\\
  & z & & \vdots\\
  & & & 0\\
  0 & \cdots & 0& a
\end{pmatrix}
B(b),
 \quad 
B(b)=\begin{pmatrix}
 1 \\
 0 & 1\\
 \vdots & & \ddots\\
 0&\ldots& b&1
 \end{pmatrix},
$$
{and let} $\tilde{\Psi}_n:\mathbb{R}^+\times \mathbb{R}\times Q_{A_{n-1}}\longrightarrow Q_{A_n}$,
$(\alpha,\beta,x)\longmapsto \tilde{\eta}$ with
$$\tilde{\eta} = \pi\left(
B(\beta)
\begin{pmatrix}
  & & & 0\\
  & x & & \vdots\\
  & & & 0\\
  0 & \cdots & 0& \alpha
\end{pmatrix}
\transp{B}(\beta)\right).
$$
{Then the} maps $\tilde{\Phi}_n$ and $\tilde{\Psi}_n$ are bijections.
\item \label{jacobianLemma}
The Jacobians of the changes of variables $y=\Phi_n(a,b,z)$ and $y=\tilde{\Phi}_n(a,b,z)$ are {given by}
 \begin{equation}
  J_{\Phi_n}(a,b,z) = a, \quad  J_{\tilde{\Phi}_n}(a,b,z) = a.
 \end{equation}
The Jacobians of the changes of variables $\eta=\Psi_n(\alpha,\beta,x)$ and $\eta=\tilde{\Psi}_n(\alpha,\beta,x)$ are {given by}
  \begin{equation}
 J_{\Psi_n}(\alpha,\beta,x) = x_{22}, \quad  J_{\tilde{\Psi}_n}(\alpha,\beta,x) = x_{n-1,n-1}.
 \end{equation}
\end{enumerate}
\end{pro}
\begin{proof}
1. Let $y' = \begin{pmatrix}
 a & 0&\hdots & 0\\
 0\\
 \vdots & & z\\
 0
\end{pmatrix}$
 and
 $\eta' = \begin{pmatrix}
 \alpha & 0&\hdots & 0\\
 0\\
 \vdots & & x\\
 0
\end{pmatrix}$.
 Then
  \begin{equation}\label{yij}
  y_{ij}=\begin{cases}
           ab \quad \textnormal{if}\; (i,j)=(1,2) \, \textnormal{or}\; (i,j)=(2,1),\\
           ab^2+z_{22} \quad\textnormal{if} \; i=j=2,\\
           y'_{ij} \quad \textnormal{otherwise}.
          \end{cases}
          \end{equation}
Thus, on the one hand, if $(a,b,z)\in \mathbb{R}^+\times \mathbb{R}\times P_{A_{n-1}}$, then
          $y\in Z_{A_n}$.
{And} $z > 0 $ implies $y'> 0$ as every principal minor of $y'$ equals $a$ times a  principal minor of $z$. 
From $y=T y' \transp{T}$ with $T=A(b)$,
 we get $y \in P_{A_n}$.
          On the other hand, if $y\in P_{A_n}$,
          we have $a=y_{11}>0$, $b=\frac{y_{12}}{y_{11}}$, $z_{22}=y_{22}-\frac{y_{12}^2}{y_{11}}$
          and
          $z_{ij}=y_{ij}$ for all $i\neq 2$ and $j\neq 2$. We use the notation $z=(z_{ij})_{2\le i,j\le n}$.
        Now, let us show that $z \in P_{A_{n-1}} $.
         We have $y'=T^{-1} y\, \transp{T}^{-1} > 0$.
   Hence, we have also $z>0$ since each principal minor of $z$ {equals} $1/a$ times a principal
   minor of $y'$.
  Therefore, the map $\Phi_n$ is indeed  a bijection from
  $\mathbb{R}^+\times \mathbb{R}\times P_{A_{n-1}}$ {onto} $P_{A_n}$.

{Let us turn to $\Psi_n$.
The relation between $\eta$ and $\eta'$ is given by}
  \begin{equation}\label{etaij}
  \eta_{ij}=\begin{cases}
	  \alpha +\beta^2 x_{22} \quad\textnormal{if} \; i=j=1,\\
          \beta x_{22} \quad \textnormal{if}\; (i,j)=(1,2) \, \textnormal{or}\; (i,j)=(2,1),\\
           \eta'_{ij} \quad \textnormal{otherwise}.
          \end{cases}
          \end{equation}
    First we show that  if
        $(\alpha,\beta,x)\in \mathbb{R}^+\times \mathbb{R}\times Q_{A_{n-1}}$, then
          $\eta\in I_{A_n}$.
         Actually, since $x_{\{2,3\}}>0$, we have
          $\alpha +\beta^2 x_{22} >0$ and
          $\eta_{\{1,2\}}=\begin{pmatrix}
                        \alpha +\beta^2 x_{22} & \beta x_{22}\\\beta x_{22}&  x_{22}
                       \end{pmatrix}>0$. 
          On the other hand, if $\eta\in Q_{A_n}$,
          we have
                   $x_{ij}=\eta_{ij}$ for all $i,j = 2,\hdots, n$. 
Thus, $\eta\in Q_{A_n}$
          implies $x\in Q_{A_{n-1}}$.

\noindent
2. {Let $\tilde{y}' = \begin{pmatrix}
  & & & 0\\
  & z & & \vdots\\
  & & & 0\\
  0 & \hdots & 0& a
\end{pmatrix}$ and $\tilde{\eta}'=\begin{pmatrix}
  & & & 0\\
  & x & & \vdots\\
  & & & 0\\
  0 & \hdots & 0& \alpha
\end{pmatrix} $.
{Then we have}
  \begin{equation}\label{tildeyij}
 \tilde{ y}_{ij}=\begin{cases}
           ab \quad \textnormal{if}\; (i,j)=(n-1,n) \, \textnormal{or}\; (i,j)=(n,n-1),\\
           ab^2+z_{n-1,n-1} \quad\textnormal{if} \; i=j = n-1,\\
           \tilde{y}'_{ij} \quad \textnormal{otherwise},
          \end{cases}
          \end{equation}    
 {and}        
  \begin{equation}\label{tildeetaij}
  \tilde{\eta}_{ij}=\begin{cases}
	  \alpha +\beta^2 x_{n-1,n-1} \quad\textnormal{if} \; i=j=n,\\
          \beta x_{n-1,n-1} \quad \textnormal{if}\; (i , j)=(n-1 , n) \, \textnormal{or}\; (i , j)=(n , n-1),\\
           \tilde{\eta}'_{ij} \quad \textnormal{otherwise}.
          \end{cases}
          \end{equation}  
Similar reasoning as above shows that $\tilde{\Phi}$ and $\tilde{\Psi}$ are indeed bijections.}

\noindent 3.
 The proof is by direct computation.
\end{proof}

\begin{lem} \label{recurrenceLemma}
\begin{enumerate}
\item \label{recurrenceLemma3}{
Let $y=\Phi_n(a,b,z)$ and  $\eta=\Psi_n(\alpha,\beta,x)$.
Then, for all $M = 2,\hdots, n$,
\begin{equation}\label{eqn:recurrent_Delta}
 \Delta_{\us}^{(M)}(y) = a^{s_1}\Delta_{(s_{2},\hdots,s_n)}^{(M)}(z){,}
\end{equation}
\begin{equation}\label{eqn:recurrent_delta}
 \delta_{\us}^{(M)}(\eta) = \alpha^{s_1}\delta_{(s_{2},\hdots,s_n)}^{(M)}(x).
\end{equation}
Let $y=\tilde{\Phi}_n(a,b,z)$ and  $\eta=\tilde{\Psi}_n(\alpha,\beta,x)$.
Then, for all $M = 1,\hdots, n-1$,
\begin{equation}\label{eqn:recurrent_Delta1}
 \Delta_{\us}^{(M)}(y) = a^{s_n}\Delta_{(s_{1},\hdots,s_{n-1})}^{(M)}(z){,}
\end{equation}
\begin{equation}\label{eqn:recurrent_delta1}
 \delta_{\us}^{(M)}(\eta) = \alpha^{s_n}\delta_{(s_{1},\hdots,s_{n-1})}^{(M)}(x).
\end{equation}
}
\item 
Let us define $\varphi_{A_n} : Q_{A_n} \to \mathbb{R}_{+}$ by $\varphi_{A_1}(\eta)=\eta^{-1}$,
 {and for $n \ge 2$}
\begin{equation}\label{char}
  \varphi_{A_n}(\eta)=  \overset{n-1}{\underset{i=1}{\prod}}|\eta_{\{i,i+1\}}|^{-3/2}\underset{i\neq 1,n}{\prod} \eta_{ii}.
  \end{equation}
Let $\eta= \Psi_n(\alpha,\beta,x)$ and  $\tilde{\eta}= \tilde{\Psi}_n(\alpha,\beta,x)$.
Then,
 \begin{equation} \label{eqn:recurrent_vphi}
 \varphi_{A_n}(\eta)
 = x_{22}^{-1/2} \alpha^{-3/2} \varphi_{A_{n-1}}(x)
 \end{equation}
 and 
 \begin{equation} \label{eqn:recurrent_vphi_Bis}
 \varphi_{A_n}(\tilde{\eta})
 = x_{n-1,n-1}^{-1/2} \alpha^{-3/2} \varphi_{A_{n-1}}(x).
 \end{equation}
\item \label{traceLemma}
 If $y=\Phi_n(a,b,z)$ and $\eta=\Psi_n(\alpha,\beta,x)$,
 then
 \begin{equation}
  \tr(y\eta) = a \alpha + a x_{22}(b+\beta)^2 + \tr(zx).
 \end{equation}
 If $y=\tilde{\Phi}_n(a,b,z)$ and $\eta=\tilde{\Psi}_n(\alpha,\beta,x)$,
 then
 \begin{equation}\label{eqn:tracelemma2}
  \tr(y\eta) = a \alpha + a x_{n-1,n-1}(b+\beta)^2 + \tr(zx).
 \end{equation}

\end{enumerate}
\end{lem}
\begin{proof}
1. {For $M\geq 2$, we have
$$\frac{\Delta_{\us}^{(M)}(y)}{\Delta_{(s_{2},\hdots,s_n)}^{(M)}(z)}
= (y_{11})^{s_1-s_2}\left[\prod_{i=2}^{M-1} \left( \frac{|y_{\{1:i\}}|}{|z_{\{2:i\}}|}\right)^{s_i-s_{i+1}}\right]\left( \frac{|y|}{|z|}\right)^{s_M}
.$$
Using  Lemma \ref{a:triang},
we have $|y_{\{1:i\}}|= a |z_{\{2:i\}}|$.
Thus, $$\frac{\Delta_{\us}^{(M)}(y)}{\Delta_{(s_{2},\hdots,s_n)}^{(M)}(z)}= a^{s_1}.$$
{Noting that $a=y_{nn}$, we have for $M=1, \dots, n-1$,} 
\begin{eqnarray*}
 \Delta_{\us}^{(M)}(\tilde{y})&=& |\tilde{y}|^{s_1}\prod_{i=2}^{n} |\tilde{y}_{\{i:n\}}|^{s_i-s_{i-1}}
 = a^{s_1}|z|^{s_1} \prod_{i=2}^{n-1} \left(a\,|z_{\{i:n\}}|^{s_i-s_{i-1}}\right)a^{s_n-s_{n-1}}\\
&=&a^{s_n} |z|^{s_1} \prod_{i=2}^{n-1} |z_{\{i:n\}}|^{s_i-s_{i-1}} = a^{s_n} \Delta_{(s_1,\hdots,s_{n-1})}^{(M)}(z).
 \end{eqnarray*}
}
Similarly, we show that $\delta_{\us}^{(M)}(\eta)=\alpha^{s_1}\delta_{(s_{2},\hdots,s_n)}^{(M)}(x)$ for $M\geq 2$
and  that 
$\delta_{\us}^{(M)}(\eta)
= \alpha^{s_n}   \delta_{\us}^{(M)}(x)$
for all $M \leq n-1$.

\noindent 2.
 {Let $\eta = \Psi(\alpha,\beta,x)$ and $\tilde{\eta} = \tilde{\Psi}(\alpha,\beta,x)$.
For $n=2$, we have
 \begin{align*}\varphi_{A_2}(\eta)&= |\eta_{\{1,2\}}|^{-3/2} = \begin{vmatrix}
                                                   \alpha +\beta^2 x & \beta x\\ \beta x & x
                                                  \end{vmatrix}^{-3/2}
 = \alpha^{-3/2}x^{-3/2}\\
 &= x^{-1/2} \alpha^{-3/2}\varphi_{A_1}(x).\end{align*}
For $n>2$, using \eqref{etaij}, we have
 \begin{eqnarray*}
 \varphi_{A_n}(\eta)&=& \eta_{22}\,|\eta_{\{1,2\}}|^{-3/2}\,\frac{\overset{n-1}{\underset{i=2}{\prod}}|\eta_{\{i,i+1\}}|^{-3/2}}{\overset{n-1}{\underset{i=3}{\prod}} \eta_{ii}^{-1}}
 = x_{22}^{-1/2} \alpha^{-3/2} \varphi_{A_{n-1}}(x).
 \end{eqnarray*}
The proof of the second part is analogous.

\noindent 3.
The proof is by direct computation.
}
 \end{proof}
\section{Laplace transform of generalized power functions on $Q_G$ and $P_G$}\label{LaplaceSection}
\begin{thm}\label{laplacedelta}

  For all $n\geq 1$, {for all\, $1\leq M \leq n$\,} and for all  $y\in P_{A_n}$,
 the integral\\
 $ \int_{Q_{A_n}}  e^{-\tr(y\eta)}\delta_{\us}^{(M)}(\eta)\varphi_{A_n}(\eta)d\eta$ converges if
 and only if $s_i >\frac{1}{2}$ for all $i\neq M$, and $s_M>0$. In this case,
 we have
  \begin{equation}\label{eqn:laplacedelta}
\int_{Q_{A_n}} e^{-\tr(y\eta)}\delta_{\us}^{(M)}(\eta)\varphi_{A_n}(\eta)d\eta
= \pi^{(n-1)/2}
  \Bigl\{\prod_{i\neq M}\Gamma(s_i-\frac{1}{2})\Bigr\}
 \Gamma(s_M)
  { \Delta^{(M)}_{-\us}(y).}
\end{equation}
\end{thm}
\begin{proof}
We will proceed by induction on the number $n$ of vertices.
For $n=1$,  we have  the gamma integral that converges if and only if $s>0$, {so that}
 \begin{equation*}
 \int_{0}^{\infty} e^{-y \eta} \delta_s^{(1)}(\eta)\varphi_{A_1}(\eta)d\eta =
  \int_{0}^{\infty} e^{-y \eta} \eta^{s-1} d\eta= \Gamma(s) y^{-s}.
\end{equation*}

Now assume that the assertion holds for {a graph with $n-1$ vertices}.\\
{\it Case $M > 1$}. Let $y=\Phi_n(a,b,z)$ and let us make the change of variable $\eta=\Psi_n(\alpha,\beta,x)$.
 The induction hypothesis gives
  \begin{align}\label{eqn:induction}
   &\int_{Q_{A_{n-1}}} e^{-\tr(zx)}\delta_{(s_2,\hdots,s_n)}^{(M)}(x)\varphi_{A_{n-1}}(x)dx\\
   &\qquad=\quad \pi^{(n-2)/2}
  \Bigl\{\prod_{i\neq 1,M}\Gamma(s_i-\frac{1}{2})\Bigr\}
 \Gamma(s_M)
  \Delta^{(M)}_{-(s_2,\hdots,s_n)}(z)\nonumber
  \end{align}
  if and only if $s_i >\frac{1}{2}$ for all $i\neq M$, and $s_M>0$.
  By Lemma \ref{jacobianLemma}, the change of variable  $\eta=\Psi_n(\alpha,\beta,x)$
  gives $d\eta=x_{22}d\alpha d\beta dx$. Thus, we have
  \begin{eqnarray*}
  & &\int_{Q_{A_n}} e^{-\tr(y\eta)}\delta_{\us}^{(M)}(\eta)\varphi_{A_n}(\eta)d\eta\\
    & & \; =  \int_{0}^{\infty} \int_{-\infty}^{\infty} \int_{Q_{A_{n-1}}}
  e^{-(a \alpha + a x_{22}(b+\beta)^2 + \tr(zx))} \alpha^{s_1-3/2}\\
  && \hspace{3cm}\times \quad \delta_{(s_2,\hdots,s_n)}^{(M)}(x)\varphi_{A_{n-1}}(x)x_{22}^{1/2}
  d\alpha d\beta dx,
  \end{eqnarray*}
  where we used parts \ref{traceLemma} and \ref{recurrenceLemma3}
  of   Lemma \ref{recurrenceLemma}. 
Now, using the Gaussian integral
\begin{equation*}
 \int_{-\infty}^{\infty} e^{-a x_{22}(b+\beta)^2}d\beta  = \pi^{1/2}a^{-1/2}x_{22}^{-1/2}
\end{equation*}
and the gamma integral
\begin{equation*}
  \int_{0}^{\infty} e^{-a \alpha} \alpha^{s_1-3/2}d\alpha = a^{-s_1+1/2}\Gamma(s_1-\frac{1}{2}),
\end{equation*}
that is finite if and only if $s_1>\frac{1}{2}$,
we get
\begin{equation}\label{eqn:recurrent_integral}
\begin{aligned}
 {}&\int_{Q_{A_n}} e^{-\tr(y\eta)}\delta_{\us}^{(M)}(\eta)\varphi_{A_n}(\eta)d\eta\\
&= \pi^{1/2}a^{-s_1}\Gamma(s_1-\frac{1}{2})
  \int_{Q_{A_{n-1}}} e^{-\tr(zx)} \delta_{(s_2,\hdots,s_n)}^{(M)}(x)\varphi_{A_{n-1}}(x)dx.
\end{aligned}
\end{equation}
Finally, using {Formulas} \eqref{eqn:induction} and \eqref{eqn:recurrent_Delta} completes the proof
in the case $M>1$.\\
{\it Case $M=1$.} Let $y=\tilde{\Phi}_n(a,b,z)$ and let us make the change of variable $\eta=\tilde{\Psi}_n(\alpha,\beta,x)$.
The proof is similar.
\end{proof}

\begin{thm}\label{laplaceDelta}

  For all $n\geq 1$, {for all \,$1\leq M \leq n$\,} and\, for all  $\eta\in Q_{A_n}$,
 the integral\\
 $ \int_{P_{A_n}}  e^{-\tr(y\eta)}\Delta_{\us}^{(M)}(y)dy$ converges if
 and only if $s_i > -\frac{3}{2}$ for all $i\neq M$, and $s_M>-1$. In this case,
 we have
  \begin{equation}\label{FORM_LaplaceDelta}
\int_{P_{A_n}} e^{-\tr(y\eta)}\Delta_{\us}^{(M)}(y)dy
= \pi^{(n-1)/2}
  \Bigl\{\prod_{i\neq M}\Gamma(s_i+\frac{3}{2})\Bigr\}
 \Gamma(s_M+1)
  \delta^{(M)}_{-\us}(\eta)\varphi_{A_n}(\eta).
\end{equation}
\end{thm}

\begin{proof}Similar to the proof of Theorem \ref{laplacedelta} using Proposition \ref{changeOfVar} and Lemma \ref{recurrenceLemma}.
 \end{proof}

On a convex cone $\Omega$ we define the characteristic function $\varphi_\Omega$ of the cone
as the Laplace transform of the Lebesgue measure of the dual cone.
The measure $\varphi_\Omega(x) dx$ is called the canonical measure of $\Omega$.
It is invariant by the linear automorphisms of $\Omega$ \cite[]{faraut1994}.

\begin{cor}\label{canonical}
$\varphi_{Q_{A_n}}= \mathrm{const}\,.\,\varphi_{{A_n}}$.
\end{cor}
\begin{proof}
 The result,\, 
 $\label{eqn:canonical}
  \left(\frac{4}{\pi^2}\right)^{\frac{n-1}{2}}\int_{P_{A_n}} e^{-\tr(y\eta)} dy\,=\, \varphi_{A_n}(\eta)
 $,
 is obtained by substituting $\us = (0,\hdots,0)$ into Theorem \ref{laplaceDelta}.
\end{proof}
\begin{rem}
 Formulas {\eqref{eqn:laplacedelta} and \eqref{FORM_LaplaceDelta}} may seem similar but in 
 \eqref{FORM_LaplaceDelta} the integrand does not contain the characteristic function of the cone $P_{A_n}$. 
This function 
 is unknown except for $A_4$ when it is not a power function \cite[Prop.3.2]{L-M}.
\end{rem}

\section{Wishart exponential families on $Q_G$}\label{RieszWishartQ}

Let us define the Riesz measure $R_{\us}^{(M)}$ on $Q_G$ by 
\begin{equation}
 dR_{\us}^{(M)}(x)= C_{\us} \delta_{\us}^{(M)}(x)\varphi_{A_n}(x)1_{Q_{A_n}}(x) dx,
\end{equation}
where $C_{\us}^{-1} = \pi^{(n-1)/2}\left(\underset{i\neq M}{\prod}\Gamma(s_i-\frac{1}{2}) \right) \Gamma(s_M)$. 
Therefore, from Theorem \ref{laplacedelta}, the Laplace transform of the measure $dR_{\us}^{(M)}$ is given 
for all $s_i >\frac{1}{2}$,   $i\neq M$ and $s_M>0$ by 
\begin{equation}\label{LapRieszQ}
 \mathcal{L}(R_{\us}^{(M)} )(y)= \int_{Q_{A_n}} e^{-\tr(y\eta)}dR_{\us}^{(M)}(\eta) =  \Delta^{(M)}_{-\us}(y),\ \ \ y\in P_{A_n}.
\end{equation}
Wishart natural exponential family $\gamma_{\us,y}^{(M)}$ on $Q_G$ is, by definition,  generated by 
the Riesz measure $dR_{\us}^{(M)}$.
The density function of the Wishart {distribution} on $Q_G$ is given by
\begin{equation}
 \gamma_{s,y}^{(M)}(dx)= C_{\us} e^{-\tr(yx)}\Delta^{(M)}_{\us}(y) \delta_{\us}^{(M)}(x)\varphi_{A_n}(x)1_{Q_{A_n}}(x) dx.
\end{equation}
The Laplace transform of $\gamma_{\us,y}^{(M)}(dx)$  is 
\begin{equation*}
 \mathcal{L}(\gamma_{\us,y}^{(M)})(z) = \frac{\mathcal{L}(R_{\us}^{(M)} )(z +y)}{ \mathcal{L}(R_{\us}^{(M)} )(y)}=
 \frac{\Delta^{(M)}_{-\us}(z+y)}{\Delta^{(M)}_{-\us}(y)}.
\end{equation*}
The family
$ \gamma_{\us,y}^{(M)}$ does not depend on the normalization of the Riesz measure.

\subsection{{Mean and covariance of the Wishart distributions on $Q_G$}}

In this {subsection} we derive a formula for the  mean  of the Wishart exponential family on the cones $Q_G$.
It is known from the general theory of exponential families of distributions,  that the mean of $\gamma^{(M)}_{\us,y}$ 
is obtained by differentiation with respect to $y$ of 
the Laplace transform of the Riesz measure:
\begin{equation}\label{meanOne}
m^{(M)}_{\us}(y)= -{\rm grad}_y \log \Delta^{(M)}_{-\us}(y) \in Q_G.
\end{equation}


For all matrix $A$ in $Z_G$  and a subset $B\subset V$ of the set of vertices $V$ of $G$ we note 
$(A_B)^0$ the matrix in $Z_G$  such that 
$(A_B)^0_{ij}= \begin{cases}
  A_{ij} & \quad \textnormal{if} \quad {i,j\in V}, \\
  0 & \textnormal{otherwise}.
 \end{cases}
$

\begin{pro}\label{PROPmsM}
The mean function of the Wishart family $\gamma_{\us,y}^{(M)}$ on $Q_G$ is equal {to}
 \begin{align}\label{msM}
  {}& m_{\us}^{(M)}(y)= \\
  &\pi\left(\sum_{i=1}^{M-1} (s_i-s_{i+1}) [(y_{\{1:i\}})^{-1}]^0 + s_M y^{-1}+
	\sum_{i=M+1}^{n} (s_{i}-s_{i-1}) [ (y_{\{i:n\}})^{-1}]^0\right). \nonumber
 \end{align}
\end{pro}

\begin{proof}
Use formulas \eqref{P(M)}, \eqref{meanOne} and\, $\grad \log |y_{A}|= \left((y_A)^{-1}\right)^0$.

\end{proof}

\begin{pro}\label{LEMkappa}
For all $y\in P_G$, {we have} 
$$
\langle m^{(M)}_{\us}(y), y \rangle=\kappa(\us),
$$
where the constant $\kappa(\us)$ {is} $\sum_{i=1}^n s_i -(n-M)s_M$.
\end{pro}
\begin{proof}
Observe that by \eqref{P(M)}, for any $c>0$,\,
$\Delta_{-\us}^{(M)}(cy)= c^{-\kappa(\us)}\Delta_{-\us}^{(M)}(y).$
By \eqref{meanOne},
$
\langle m^{(M)}_{\us}(y) , y \rangle= - \langle {\rm grad}_y \log \Delta^{(M)}_{-\us}(y)  , y \rangle
$.
Set $F(y)=\log \Delta^{(M)}_{-\us}(y)$. By the  chain rule,
$
\langle {\rm grad}_y F(y)  , y \rangle= \frac{d}{dt}F(ty)\big\vert_{t=1}. 
$
The map $t\rightarrow F(ty)= \log \varphi(t), \R^+\rightarrow \R,$
where $\varphi(t)=\Delta^{(M)}_{-\us}(ty)$, satisfies $\varphi(ct)=  c^{-\kappa(\us)} \varphi(t)$.
Hence $\varphi(c)=  c^{-\kappa(\us)} \varphi(1)$ and 
$
\frac{d}{dt}F(ty)\big\vert_{t=1}= \frac{\varphi'(1)}{\varphi(1)}=-\kappa(\us).
$
Thus $\langle {\rm grad}_y F(y)  , y \rangle=-\kappa(\us)$ and the result follows.
\end{proof}

Differentiating the mean function gives the covariance {function}. 
{For $A \in S_n$, let $\PP(A): Z_G \rightarrow {I_G}$ be the quadratic operator defined by $\PP(A)u ={\pi}(AuA),\,\,\,u \in Z_G$}.

\begin{pro}
The covariance function of the Wishart family $\gamma_{\us,y}^{(M)}$ on $Q_G$ is equal
 \begin{align}\label{covariance}
   v(y)&=-{m'}_{\us}^{(M)}(y)= \sum_{i=1}^{M-1} (s_i-s_{i+1})\PP\left[ \left((y_{\{1:i\}})^{-1}\right)^0\right] \,
   + s_M \PP(y^{-1})\\
   & \,+    \sum_{i=M+1}^{n} (s_i-s_{i-1})\PP\left[ \left((y_{\{i:n\}})^{-1}\right)^0\right].\nonumber
  \end{align}
  \end{pro}
\subsection{{Inverse mean map}}\label{SecInverseMean}
 In the study of the exponential family
 $(\gamma^{(M)}_{\us,y})_{y\in P_G}$ it is important to determine explicitly 
 {the inverse of} the mean map
$
\psi^{(M)}_{\us}:\  m=m^{(M)}_{\us}(y)\mapsto y,
$
{which we refer to {as} the inverse mean map in the sequel.}
The following theorem is known {for Wishart exponential families}  on homogeneous cones \cite[]{ishiHammamet}.
Surprisingly, it is also true on $Q_G$.
\begin{thm}\label{inverseMean}
 The inverse mean map $\psi^{(M)}_{\us}$ is given by the formula
\begin{equation}\label{inverse}
\psi^{(M)}_{\us}(m)={\rm grad}_m \log\delta_{\us}^{(M)}(m),\ m\in Q_G.
\end{equation}
\end{thm}

The proof consists in  following  steps: \\
1. One shows that there exists a constant $c_{\us}$ depending only on $\us$  such that for any $y\in P_G$
$$
  \delta^{(M)}_{\us}(m^{(M)}_{\us}(y))={c_{\us} \Delta^{(M)}_{-\us}(y)=}  c_{\us}\delta^{(M)}_{\us}(\pi(y^{-1}) ).
$$
This is done in Proposition \ref{delta_m_s} below.
\\
2. One uses a differential calculus argument, based  on the Legendre transform methods. 

\begin{pro}\label{delta_m_s}
The following formula holds  for any $y\in P_G$  and $\us\in \mathbb{R}^n$:
$$
 \delta^{(M)}_{\us}(m^{(M)}_{\us}(y))= {\left(\prod_{i=1}^n s_i^{s_i}\right) \Delta^{(M)}_{-\us}(y)}= \left(\prod_{i=1}^n s_i^{s_i}\right) \; \delta^{(M)}_{\us}(\pi(y^{-1}) ).
$$
\end{pro}



The proof of Proposition \ref{delta_m_s}  will need a generalization of Lemma  \ref{determinant}, where coefficients of inverse  matrices
of principal submatrices $y_{\{1:k\}}$ (or of $y_{\{k:n\}}$) are simultanously considered.
Define for $y\in P_G$,\,
$
\eta^{(k)}=(y_{\{1:k\}})^{-1},\ \ \ \eta^{[k]}=(y_{\{k:n\}})^{-1}. 
$
The rows and the columns of the matrix $\eta^{(k)}$ are numbered by $i=1,\ldots,k$
and   the rows and the columns of the matrix $\eta^{[k]}$ are numbered by $i=k,\ldots,n$.
\begin{lem}\label{determinantElargi}
 Let $y\in P_G$.
\begin{enumerate}
\item For all $i\in V$ and $k,m \ge  i+1$ we have
\begin{equation}\label{detElargi}
D_i^{k,m}:= \begin{vmatrix}
  \eta^{(k)}_{ii} & \eta^{(m)}_{i,i+1}\\
\eta^{(k)}_{i,i+1} & \eta^{(m)}_{i+1,i+1}
 \end{vmatrix}= |y_{\{1:m\}}|^{-1}\vert y_{{\{1:m\}}\backslash\{i,i+1\}} \vert.
\end{equation}
\item For all $i\in V$ and $k,m \le  i<n$ we have
\begin{equation}\label{detElargiBIS}
D_i^{[k,m]}:= \begin{vmatrix}
  \eta^{[k]}_{ii} & \eta^{[m]}_{i,i+1}\\
\eta^{[k]}_{i,i+1} & \eta^{[m]}_{i+1,i+1}
 \end{vmatrix}= |y_{\{k:n\}}|^{-1}\vert y_{{\{k:n\}}\backslash\{i,i+1\}} \vert.
\end{equation}
\end{enumerate}
\end{lem}

\begin{proof}
 Similar to the proof of Lemma \ref{determinant}; instead of $y$ use $y_{\{1:k\}}$ or $y_{\{k:n\}}$.
\end{proof}

\begin{proof} (of  Proposition  \ref{delta_m_s})
We will deal with $\delta^{(M)}_{\us}(m^{(M)}_{\us}(y))=\delta^{\prec_M}_{\us}(m^{(M)}_{\us}(y))$ where the order $\prec_M$
was defined in (\ref{order0}).
By  formula (\ref{msM}) 
  and by the definition of $\delta^{\prec_M}_{\us}$ we obtain {that} $\delta^{\prec_M}_{\us}(m_{\us}(y))$ equals
$$
\prod_{i=1}^{M-1} \left(  \frac1{c_i}
            \begin{vmatrix}
             x_i+a_i&b_i\\ b_i&c_i
            \end{vmatrix}  \right)^{s_i} 
						(s_M  \eta^{(n)}_{MM})^{s_M} 
						\prod_{i=M+1}^{n} \left(  \frac1{c'_i}
						\begin{vmatrix}
             x'_i+a'_i&b'_i\\ b'_i&c'_i
            \end{vmatrix}  \right)^{s_i},
$$
where $x_i=(s_{i}-s_{i+1})  \eta^{(i)}_{ii},\,\, a_i=\sum_{k=i+1}^{M-1} (s_{k}-s_{k+1})  \eta_{ii}^{(k)}+ s_M \eta_{ii}^{(n)}$,
\begin{eqnarray*}
&& 
 b_i= \sum_{k=i+1}^{M-1} (s_{k}-s_{k+1})  \eta_{i,i+1}^{(k)}+ s_M \eta_{i,i+1}^{(n)},\\
&&c_i=\sum_{k=i+1}^{M-1} (s_{k}-s_{k+1})  \eta_{i+1,i+1}^{(k)} + s_M \eta_{i+1,i+1}^{(n)},\\
&&a'_i= \sum_{k=M+1}^{i-1} (s_{k}-s_{k-1})  \eta_{ii}^{[k]}+ s_M \eta_{ii}^{[1]},\\
&&b'_i=\sum_{k=M+1}^{i-1} (s_{k}-s_{k-1})  \eta_{i,i-1}^{[k]}+ s_M \eta_{i,i-1}^{[1]},\\
&&c'_i=  \sum_{k=M+1}^{i-1} (s_{k}-s_{k-1})  \eta_{i-1,i-1}^{[k]} + s_M \eta_{i-1,i-1}^{[1]},
\end{eqnarray*}
{and  $x'_i=(s_{i}-s_{i-1})  \eta^{[i]}_{ii}$}.
Let us first compute the factors $
            \begin{vmatrix}
             x_i+a_i&b_i\\ b_i&c_i
            \end{vmatrix}/c_i$  for $i=1,\ldots, M-1$.
						We will show that
				\begin{equation}\label{firstFactor}
					\frac1{c_i}	\begin{vmatrix}
             x_i+a_i&b_i\\ b_i&c_i
            \end{vmatrix} = s_i\eta_{ii}^{(i)},\ \ i=1,\ldots, M-1.
				\end{equation}
						We have $
  {\displaystyle \frac1{c_i}
            \begin{vmatrix}
             x_i+a_i&b_i\\ b_i&c_i
            \end{vmatrix}= x_i+ \frac1{c_i}
            \begin{vmatrix}
             a_i&b_i\\ b_i&c_i
            \end{vmatrix},}
  $
so in order to prove (\ref{firstFactor}), it is sufficient to prove that
\begin{equation}\label{CLE}
 \frac1{c_i}
            \begin{vmatrix}
             a_i&b_i\\ b_i&c_i
            \end{vmatrix}= s_{i+1}  \eta^{(i)}_{ii}.
\end{equation}
In order to prove (\ref{CLE}), we first use the multilinearity of the determinant with respect to its columns and we write,
using the notation $D_i^{k,m}$ from Lemma \ref{determinantElargi},
\begin{eqnarray*}
            \begin{vmatrix}
             a_i&b_i\\ b_i&c_i
            \end{vmatrix}
            &=&\sum_{k,m=i+1}^{M-1} (s_{k}-s_{k+1})(s_{m}-s_{m+1}) D_i^{k,m} + s_M \sum_{k=i+1}^{M-1} (s_{k}-s_{k+1}) D_i^{k,n}\\
						&+& s_M \sum_{m=i+1}^{M-1} (s_{m}-s_{m+1}) D_i^{n,m} + s_M^2  D_i^{n,n}.					
\end{eqnarray*}						

By {Part 1 of Lemma \ref{determinantElargi}} we have
$D_i^{k,m}= |y_{\{1:m\}}|^{-1}\vert y_{\{1:m\}\backslash\{i,i+1\}} \vert$, which is independent of the left index $k$. 
The last fact allows to write
\begin{eqnarray*}
            \begin{vmatrix}
             a_i&b_i\\ b_i&c_i
            \end{vmatrix}
						&=&s_{i+1} \sum_{m=i+1}^{M-1} (s_{m}-s_{m+1}) D_i^{n,m} + s_{i+1} s_M  D_i^{n,n}
\\
            &=&s_{i+1}\left(\sum_{m=i+1}^{M-1} (s_{m}-s_{m+1}) \frac{  \vert y_{\{1:m\}\backslash\{i,i+1\}}\vert}{|y_{\{1:m\}}|}
					+s_M	\frac{  \vert y_{\{1:n\}\backslash\{i,i+1\}}|}{|y|}
						\right).
\end{eqnarray*}	
We factorize the determinants $\vert y_{\{1:m\}\backslash\{i,i+1\}} \vert$ and
$\vert y_{\{1:n\}\backslash\{i,i+1\}} \vert$
 in the last sum according to Lemma
\ref{nonconnected} and we write this sum as
$$
\frac{|y_{\{ 1:i-1\}}|}{|y_{\{ 1:i\}}|} \left(\sum_{m=i+1}^{M-1} (s_{m}-s_{m+1})  \frac{ |y_{\{ 1:i\}}| |y_{\{ i+2:m\}}|}{|y_{\{1:m\}}|}
+ s_M \frac{ |y_{\{ 1:i\}}| |y_{\{ i+2:n\}}|}{|y|}\right).
$$
We have $ |y_{\{1:m\}}|^{-1} |y_{\{ 1:i\}}| |y_{\{ i+2:m\}}|=\eta^{(m)}_{i+1,i+1}$.	
By definition of $c_i$ we finally obtain
$$  \begin{vmatrix} a_i&b_i\\ b_i&c_i \end{vmatrix}
= s_{i+1}\frac{|y_{\{ 1:i-1\}}|}{|y_{\{ 1:i\}}|}c_i = s_{i+1}  \eta^{(i)}_{ii}c_i $$
and formulas (\ref{CLE}) and (\ref{firstFactor}) are proved.
		
A "mirror" proof based on {Part 2 of Lemma  \ref{determinantElargi}}  shows that
\begin{equation}\label{newFactor}
\frac1 {c'_i} 	\begin{vmatrix}x'_i+a'_i&b'_i\\ b'_i&c'_i \end{vmatrix}   
= s_i\eta_{ii}^{[i]},\ \ i=M+1,\ldots, n
\end{equation}
{and that}\ \  
$\delta^{(M)}_{\us}(m^{(M)}_{\us}(y)) 
= \prod_{i=1}^n s_i^{s_i} \prod_{i=1}^{M-1} ( \eta^{(i)}_{ii})^{s_i} ( \eta^{(n)}_{MM})^{s_M}
						\prod_{i=M+1}^{n} ( \eta^{[i]}_{ii})^{s_i}.$\\
Recall that
$$ \eta^{(i)}_{ii}=\frac{|y_{\{1:i-1\}}|}{|y_{\{1:i\}}|},\ \ \ 
   \eta^{[i]}_{ii}=\frac{|y_{\{i+1:n\}}|}{|y_{\{i:n\}}|},\ \ \  \eta^{(n)}_{MM}=\frac{|y_{\{1:M-1\}}||y_{\{M+1:n\}}|}{|y|},
$$
so we deduce, using formula (\ref{P(M)}) that
\begin{eqnarray*}
&&	\prod_{i=1}^{M-1} ( \eta^{(i)}_{ii})^{s_i} ( \eta^{(n)}_{MM})^{s_M}
 \prod_{i=M+1}^{n} ( \eta^{[i]}_{ii})^{s_i}
=\Delta^{(M)}_{-\us}(y).
\end{eqnarray*}
  Applying Theorem \ref{delta-Delta}, we see that $\delta^{(M)}_{\us}(m^{(M)}_{\us}(y))= \prod_{i=1}^n s_i^{s_i} \delta_{\us}^{(M)}(\pi(y^{-1})).$
\end{proof}
{
\begin{proof} (of Theorem \ref{inverseMean}).
By   Proposition   \ref{delta_m_s}  and   formula (\ref{meanOne}) we have for $y\in P_G$
$$
m^{(M)}_{\us}(y)= -{\rm grad}_y \log  \delta^{(M)}_{\us}(m^{(M)}_{\us}(y))= {\rm grad}_y f(y),
$$
where
$$
                  f(y)= -\log  \delta^{(M)}_{\us}(m^{(M)}_{\us}(y)).
$$
We know that $m^{(M)}_{\us}: y \mapsto m$  is a diffeomorphism. Our goal is to investigate the inverse map
$$
\psi(m)= (m^{(M)}_{\us})^{-1}(m)=  ({\rm grad} f)^{-1}(m).
$$
Using Legendre-Fenchel transform methods,  we get
$$
\psi(m)=   ({\rm grad} f)^{-1}(m)=  {\rm grad} g(m),
$$
where the function $g$ is defined by
$$
g(m)=\langle m, y \rangle - f(m)\ \ {\rm {with}}\ \ m(y)= {\rm grad} f(y).
$$
Thus {$g(m)=\langle m, y \rangle + \log  \delta^{(M)}_{\us}(m)$, so that}
$$
g(m^{(M)}_{\us}(y))=\langle m^{(M)}_{\us}(y), y \rangle + \log  \delta^{(M)}_{\us}(m^{(M)}_{\us}(y)).
$$
We use now Part 2 of Proposition \ref{LEMkappa}. It follows that
$$
\psi(m)={\rm grad} g(m)= {\rm grad}_m \log  \delta^{(M)}_{\us}(m).
$$
\end{proof}
}


\begin{cor}\label{invmeanexplicit}
 The inverse mean map  $\psi_{\us}^{(M)}: Q_G \rightarrow P_G$ is given by
\begin{eqnarray}\label{inverse2}
 &&\psi_{\us}^{(M)}(m)= \sum_{k=1}^{M-1} s_k \left((m_{\{k:k+1\}})^{-1}\right)^0 \,+\, 
 \sum_{k=M+1}^{n} s_k \left((m_{\{k-1:k\}})^{-1}\right)^0 \nonumber
 \\
 & & \; - \; \sum_{k=2}^{M-1} s_{k-1} \left((m_{\{kk\}})^{-1}\right)^0 
 \,- \,(s_{M-1}- s_M+s_{M+1}) \left((m_{\{MM\}})^{-1}\right)^0 \nonumber\\
&& \qquad \,-\quad \sum_{k=M+1}^{n-1} s_{k+1} \left((m_{\{kk\}})^{-1}\right)^0. 
 \end{eqnarray}
\end{cor}
\begin{proof}
 The result is obtained by computing the gradient of $\log\delta_{\us}^{(M)}(m)$, as indicated in  \eqref{inverse}.
We use the formula \eqref{delta(M)}.
 \end{proof}
The {Lauritzen} formula \cite[]{lauritzen1996} is an explicit formula for a bijection between $Q_G$ and $P_G$.
It states that for all $x\in Q_G$, the unique $y\in P_G$ such that $\pi(y^{-1})= x$ is given by 
 \begin{equation}\label{Lauritzen}
   y = \sum_{i=1}^{n-1} (x_{\{i:i+1\}}^{-1})^0 - \sum_{i=2}^{n-1} (x_{ii}^{-1})^0.
 \end{equation}
Setting $s_1=\hdots = s_n=1$ in formula  \eqref{msM} for the mean function, we get 
\begin{equation}
 m_{{(1,\hdots,1)}}^{(M)}(y) = \pi(y^{-1})=x.
\end{equation}
Thus, 
\begin{equation}
 \psi_{{(1,\hdots,1)}}^{(M)}\left(x\right) = y
\end{equation}
is the {Lauritzen} formula. Indeed, for
 $s_1=\hdots=s_n=1$, formula \eqref{inverse2} gives 
 \begin{equation}\label{LauritzenBIS}
   \psi_{(1,\hdots,1)}^{(M)} (m) = \sum_{i=1}^{n-1} (m_{\{i:i+1\}}^{-1})^0 - \sum_{i=2}^{n-1} (m_{ii}^{-1})^0.
 \end{equation}
{ Thus we found a new proof of the Lauritzen formula, based on the observation that the Lauritzen map is the inverse mean map
for $\us={\bf 1}=(1,1,\ldots,1)$.
{At} the same time we find an infinite number  of explicit isomorphisms from $Q_G$ onto $P_G$, given by 
}
the inverse mean maps $\psi_{\us}^{(M)}$. It is an essential   generalization of the Lauritzen formula. 
Each map 
 $\psi_{\us}^{(M)}$ is a generalized Lauritzen map.

\subsection{Variance function}\label{VarianceSection}
  \subsubsection{Properties of lower-upper $M$-triangular matrices}
Here, we define and prove basic properties of  lower-upper $M$-triangular
 matrices, that we will denote by  LU(M). They are very important  in proofs of this section.
\begin{defn}
 A matrix $T$ is said to be an LU(M) triangular  matrix if for all $i < M$,
 $T_{ij}=  0$ if $j>i$ and  for all  $i> M$, $T_{ij}=  0$ if $i>j$.
  \end{defn}

In particular,
  $T$ is an LU(n) triangular  matrix if and only if it is lower triangular, 
 and $T$ is an LU(1) triangular  matrix if and only if it is upper  triangular.
An LU(M) triangular matrix $T$ is a succession of an $M \times M$ lower triangular matrix $L=T_{\{1:M\}}$ and an $(N-M)\times (N-M)$ 
 upper triangular matrix $U=T_{\{M:n\}}$ with diagonal term $T_{MM}$ in common. We write $T=s(L,U)$.\\
    
\begin{center}
\begin{tikzpicture}[scale=0.4]
  \node at (-1.3,-1.90) {$T=$}; \draw (0,0)--(0,-2)--(2,-2)--(0,0); \draw (1.50,-1.80)--(3.80,-1.80)--(3.80,-3.80)--(1.50,-1.80);
  \draw (2.8,-0.25) ellipse (0.25cm and 0.4 cm); \draw (0.75,-3.3) ellipse (0.25cm and 0.4cm); 
  \draw[dotted] (1.71,-1.90)--(4.8,-1.90); \node at (5.5,-1.90) {$T_{MM}$};\node at (0.9,-1.5) {$L$}; \node at (3.1,-2.5) {$U$};
  \draw (0,0.3)--(-0.3,0.3)--(-0.3,-3.83)--(0,-3.83);  \draw (3.80,0.3)--(4.1,0.3)--(4.1,-3.88)--(3.80,-3.88);
   \end{tikzpicture}
 \end{center}

\begin{pro}\label{ULmultiply}
\begin{enumerate}
 \item  $s(L,U)s(L',U')=s(LL',UU')$.
 \item \label{ULmultiply3}If $s(L,U)$ is invertible, then $\left(s(L,U)\right)^{-1}$ is also an LU(M) triangular matrix and\\
$\left(s(L,U)\right)^{-1}= s(L^{-1},U^{-1})$.
\item The set of $LU(M)$ triangular matrices is a group.
\end{enumerate}
\end{pro}

\begin{proof}
\begin{enumerate}
 Part 1 is proved by block matrix multiplication. Part 2 is straightforward using Part 1. Part 3 follows from Parts 1 and 2.
 \end{enumerate}
\end{proof}

\begin{lem}\label{triang}
Let {$S$ and $T$} be LU(M) triangular $n\times n$ matrices.
\begin{enumerate}
\item\label{triang1} 
\begin{enumerate}
\item
 Let $ A= K^0$ {with} $K=A_{\{1:k\}}$.
{If} $k\leq M-1$, {then}
$\transp{S}AT=
\left(\transp{S}_{\{1:k\}}KT_{\{1:k\}}\right)^0$.
\item Let $ B=K^0$ {with} $K=B_{\{k:n\}}$.
{If} $k\geq M+1$, {then}
$\transp{S}BT=\left(\transp{S}_{\{k:n\}}KT_{\{k:n\}}\right)^{{0}}.$
\end{enumerate}
\item \label{triang2} Let $A$ be {an} $n\times n$  matrix. 
Then 
$(TA\transp{S})_{\{1:i\}}= T_{\{1:i\}}A_{\{1:i\}} \transp{S}_{\{1:i\}}$ for $i \leq M-1$, and  
$(TA\transp{S})_{\{i:n\}}=T_{\{i:n\}}A_{\{i:n\}}\transp{S}_{\{i:n\}}$ for $i\geq M+1$.
\item \label{triang3}If $T$ is invertible, then 
\begin{enumerate}
 \item $(T_{\{1:k\}})^{-1}=(T^{-1})_{\{1:k\}}$ for all $k\leq M-1$;
 \item $(T_{\{k:n\}})^{-1}=(T^{-1})_{\{k:n\}}$ for all $k\geq M+1$.
\end{enumerate}
 \end{enumerate}
 
\end{lem}
\begin{proof}
 Part 1 is straightforward using block matrix multiplication and Part 1 of Lemma \ref{a:triang} in Appendix;
 for Part 2 just imagine which lines and columns intervene in the computation); Part 3
 follows from Part \ref{ULmultiply3} of Proposition \ref{ULmultiply} and Part 3 of Lemma \ref{a:triang}.
\end{proof}

\begin{pro}
 For all $y\in P_{A_n}$, for all $1\leq M\leq n$, there exists an LU(M) triangular matrix $T$ satisfying $T_{ij}=0$ if $i\not\sim j$ and such that $y={T\, \transp{T}}$.
\end{pro}
\begin{proof}
 We will proceed by induction on $n$.
The statement is obviously true for $n=1$.
Let us assume that the statement is true for $n-1$.
  Let $y\in P_{A_n}$ and $M\neq 1$.
  Let us write $y=\Phi_n(a,b,z)$ with $z\in P_{A_{n-1}}$. The induction assumption implies that there exists
  $V$ an $(n-1)\times (n-1)$ LU(M) triangular matrix such that $V_{ij}=0$ if $i\not\sim j$ and such that $z=V\,\transp{V}$.  
  Let us write 
  \[T=
\begin{pmatrix}
 1 \\
 b & 1\\
 \vdots & & \ddots\\
 0&\ldots& 0&1
\end{pmatrix}
\begin{pmatrix}
 \sqrt{a} & 0&\hdots & 0\\
 0\\
 \vdots & & {V}\\
 0
\end{pmatrix}
=\begin{pmatrix}
 \sqrt{a} & 0&\hdots & 0\\
 \sqrt{a}b\\
 \vdots & & {V}\\
 0
\end{pmatrix}.\]
$T$ is LU(M) triangular satisfying $T_{ij}=0$ if $i\not\sim j$ and $y=T\, \transp{T}$.

For $M=1$, we use $y=\tilde \Phi_n(a,b,z)$ with $z\in P_{A_{n-1}}$. 
\end{proof}

\subsubsection{ Two formulas for the variance function}

Let $m\in Q_G$.
We note  $\hat m\in S^+_n$ the unique symmetric positive definite matrix verifying
${\pi}(\hat m)=m,\ \ \hat m^{-1}\in P_G.$
Define $y=\psi_{\us}^{(M)}(m)\in P_G$. 
Decompose $y = {T\, \transp{T}}$, with $T$ an LU(M) triangular matrix such that  $T_{ij}=0$ when  $i\not\sim j$.
\begin{lem}\label{hatm}
We have
\begin{equation}
\hat m=\;\transp{T}^{-1} \begin{pmatrix}
 s_1&\ldots&0\\&\ddots&\\ 0&\ldots&s_n\end{pmatrix}
T^{-1}. 
\end{equation}
\end{lem}
\begin{proof}
Note that $y=\psi_{\us}^{(M)}(m)$ is equivalent to $m = m_{\us}^{(M)}(y)$.
The formula of the mean function 
\eqref{msM}    gives $m=\pi(Z)$, where
\begin{equation}
  Z=\sum_{i=1}^{M-1} (s_i-s_{i+1}) [(y_{\{1:i\}})^{-1}]^0 + s_M y^{-1}+
	\sum_{i=M+1}^{n} (s_{i}-s_{i-1}) [ (y_{\{i:n\}})^{-1}]^0.
\end{equation}
Using Part \ref{triang2} of  Lemma \ref{triang},
 we have 
$y_{\{1:i\}}=T_{\{1:i\}}I_{\{1:i\}}  \,^tT_{\{1:i\}}$ {for $i\leq M-1$}.
{By} Part \ref{triang3} of Lemma \ref{triang}, we get 
$(y_{\{1:i\}})^{-1}= \,^t(T^{-1})_{\{1:i\}}I_{\{1:i\}}(T^{-1})_{\{1:i\}}$.
Finally, using Part \ref{triang1} of Lemma \ref{triang}, we obtain
\begin{equation}\label{yi0}
 [(y_{\{1:i\}})^{-1}]^0 = \,^tT^{-1}(I_{\{1:i\}})^0T^{-1}, \quad {i \leq M-1}.
\end{equation}
Similarly, we have
\begin{equation}\label{yn0}
 [(y_{\{i:n\}})^{-1}]^0 = \,^tT^{-1}(I_{\{i:n\}})^0T^{-1}, \quad i\geq M+1. 
\end{equation}
Thus, 
\begin{align*}
Z&= \transp{T}^{-1}\left(\sum_{i=1}^{M-1} (s_i-s_{i+1})(I_{\{1:i\}})^0 \, + s_M I + \,\sum_{i=M+1}^{n} (s_{i}-s_{i-1}) ( I_{\{i:n\}})^0 \right)T^{-1}\\
&=\transp{T}^{-1} \begin{pmatrix}
 s_1&\ldots&0\\&\ddots&\\ 0&\ldots&s_n\end{pmatrix} T^{-1}.
\end{align*}
Therefore, $Z$ is positive definite and 
$
 Z^{-1}= T \begin{pmatrix}
 s_1^{-1}&\ldots&0\\&\ddots&\\ 0&\ldots&s_n^{-1}\end{pmatrix} \transp{T} \in P_{A_n}.
$
Indeed, for all $i<i+1<j$, {we have}
$(Z^{-1})_{ij}=\sum_{k=1}^n T_{ik}T_{jk}s_{k}^{-1}.$
Since $T_{ik}=0$ for $|k-i|>1$,
$(Z^{-1})_{ij}= T_{i,i-1}T_{j,i-1}s_{i-1}^{-1}+ T_{ii}T_{ji}s_{i}^{-1}+T_{i,i+1}T_{j,i+1}s_{i+1}^{-1}.$
But since  $|j-i|>1$, we have $T_{j,i-1}=0=T_{ji}$ and
$(Z^{-1})_{ij}= T_{i,i+1}T_{j,i+1}s_{i+1}^{-1}$.
Now since $T$ is LU(M), we have $T_{i,i+1}T_{j,i+1}=0$. In fact, $T_{i,i+1}=0$ for $i\leq M-1$ and $T_{j,i+1}=0$ for $i\geq M$.
In conclusion, we have shown that $m=\pi(Z)$ with $Z^{-1} \in P_{A_n}$, which implies $Z=\hat{m}$.
\end{proof}  

The following Proposition  derives the  formula  for the variance function $V(m)$ which, 
 for each fixed $m\in Q_G$ is a continuous operator
$V(m): Z_G\rightarrow {I_G}$  \cite[]{casalis1996}.
Recall that $\PP(A): Z_G \rightarrow {I_G}$ is the quadratic operator defined by $\PP(A)u={\pi}(AuA)$.
For $A, B \in S_n$, let 
$\PP(A,B)u=\frac{1}{2}{\pi}(AuB+BuA)$.
For all $m\in Q_G$ and $I \subset V$, we note 
\begin{equation}\label{MA}
M_I=[((\hat m^{-1})_I)^{-1}]^0.
\end{equation}
\begin{pro}\label{ProCompli}
	The variance function {$V(m)$} of a Wishart NEF on $Q_G$ is {equal to}
\begin{align}\label{varcomplic} 
{}& \, \sum_{i=1}^{M-1} (s_i-s_{i+1}) \PP\left( \sum_{j=1}^{i-1} \left(\frac1{s_j}-\frac1{s_{j+1}}\right)M_{\{1:j\}} + \frac1{s_i} M_{\{1:i\}}\right) \nonumber\\
&+ \,s_M \PP\left(  \frac{ \hat{m}}{s_M}+\sum_{j=1}^{M-1} \left(\frac1{s_j}-\frac1{s_{j+1}}\right){M}_{\{1:j\}} + 
\sum_{k=M+1}^{n} \left(\frac1{s_k}-\frac1{s_{k-1}}\right)M_{\{k:n\}}\right) \nonumber\\
&+ \, \sum_{i=M+1}^{n} (s_i-s_{i-1}) \PP\left(\frac1{s_i} M_{\{i:n\}}+ \sum_{j=i+1}^{n} \left(\frac1{s_j}-\frac1{s_{j-1}}\right)M_{\{j:n\}}\right). 
\end{align}
\end{pro}
 \begin{proof}
  The variance function is given for all $m\in Q_{A_n}$ by $V(m)=v(\psi_{\us}^{(M)}(m))$, where $v(y)$ is given by
  \eqref{covariance}.   
Let $y=\psi_{\us}^{(M)}(m)=T\,^tT$, where $T$ is LU(M). 
From Lemma \ref{hatm}, we have
\[
 \hat m^{-1}=\; T \begin{pmatrix}
 s_1^{-1}&\ldots&0\\&\ddots&\\ 0&\ldots&s_n^{-1}\end{pmatrix}
\,^tT.
\] 

Using Lemma \ref{triang},
we get
\begin{equation}
 M_{\{1:i\}}=  \,^tT^{-1} \left(\mathop{diag}(s_1,\hdots, s_i)\right)^0  T^{-1}, \quad i\leq M-1
\end{equation}
and
\begin{equation}
 M_{\{i:n\}}=  \,^tT^{-1} \left(\mathop{diag}(s_i,\hdots, s_n)\right)^0  T^{-1}, \quad i\geq M+1.
\end{equation}
Thus,  for all\; $2\leq i\leq M-1$, {we have}
\begin{align}\label{basei}
 \frac{1}{s_1} M_1= \,^tT^{-1} e_1 T^{-1}, \quad  \frac{1}{s_i}( M_{\{1:i\}}- M_{1:i-1)}= \,^tT^{-1} e_i T^{-1}, \end{align}
 and for all  $n-1 \geq i \geq M+1$, {we have}
\begin{align}\label{basen}
 \frac{1}{s_n} M_n= \,^tT^{-1} e_n T^{-1}, \quad \frac{1}{s_i} (M_{\{i:n\}}-M_{i+1:n})= \,^tT^{-1} e_i T^{-1},
\end{align}
where $e_i$ is the matrix with $e_{ii}=1$ and $e_{ij}=0$ for all $i\neq j$.
Observing that 
$(I_{\{1:i\}})^0= \sum_{k=1}^i e_i \quad \textnormal{and}\; (I_{\{i:n\}})^0= \sum_{k=i}^n e_i,$
{and} using \eqref{yi0} and \eqref{basei}, {we obtain} for $i\leq M-1$
 \begin{align*}
  [(y_{\{1:i\}})^{-1}]^0  &= \,^tT^{-1}(I_{\{1:i\}})^0T^{-1} = \,^tT^{-1} \left(\sum_{k=1}^i e_i \right) T^{-1} = \sum_{k=1}^i \left( \,^tT^{-1}  e_i  T^{-1}\right)\\
 &= \frac{1}{s_1} M_{\{1\}} + \frac{1}{s_2} (M_{\{1:2\}}-M_{\{1\}})+ \hdots + \frac{1}{s_i}(M_{\{1:i\}}-M_{\{1:i-1\}})\\
  &= \left(\frac{1}{s_1}-\frac{1}{s_2} \right) M_{\{1\}} + \hdots + \left(\frac{1}{s_{i-1}}-\frac{1}{s_i} \right) M_{\{1:i-1\}} +  \frac{1}{s_i}M_{\{1:i\}}.
 \end{align*}
Similarly, using \eqref{yn0} and \eqref{basen}, we obtain for $i\geq M+1$,
\begin{align*}
 [(y_{\{i:n\}})^{-1}]^0 
 &= \frac{1}{s_i} M_{\{i:n\}}+ \left(\frac{1}{s_{i+1}}-\frac{1}{s_i} \right) M_{\{i+1:n\}} + \hdots + \left(\frac{1}{s_{n}}-\frac{1}{s_{n-1}} \right) M_{\{n\}}.
\end{align*}
We also observe that 
\begin{equation}\label{eM}
\,^tT^{-1}e_{M} T^{-1}=\frac{1}{s_M}\left(\hat{m}-M_{\{1:M-1\}}-M_{\{M+1:n\}}\right).
\end{equation}
Thus, by \eqref{basei}, \eqref{basen} and \eqref{eM}, {we get}
\begin{align}\label{y-1}
       y^{-1}&= \sum_{i=1}^{n} \,^tT^{-1}e_{i} T^{-1}
       = \sum_{i=1}^{M-1} \,^tT^{-1}e_{i} T^{-1} + \,^tT^{-1}e_{M} T^{-1} + \sum_{i=M+1}^{n} \,^tT^{-1}e_{i} T^{-1}\nonumber\\
 &= \frac{ \hat{m}}{s_M}+\sum_{j=1}^{M-1} \left(\frac1{s_j}-\frac1{s_{j+1}}\right){M}_{\{1:j\}} + 
\sum_{j=M+1}^{n} \left(\frac1{s_j}-\frac1{s_{j-1}}\right)M_{\{j:n\}}.
\end{align}
Substituting these expressions of $[(y_{\{1:i\}})^{-1}]^0$, $y^{-1}$ and $ [(y_{\{i:n\}})^{-1}]^0$\,  into $v(y)$ given by   \eqref{covariance}, we obtain the stated result.
 \end{proof}

We prove now a much simpler  formula for the variance function on $Q_G$, surprisingly  similar 
 to  the variance function on a homogeneous cone, in particular on the symmetric cone
 $S^+_n$ ({cf.} \cite{GIK}).
 
\begin{thm}\label{ThNice}
  The variance function of the Wishart exponential family $\gamma_{\us,y}^{(M)}$ is
 \begin{eqnarray} \label{mainVar}
& &V(m)= \left(\frac{1}{s_1}+\frac{1}{s_n}-\frac{1}{s_M}\right)\PP(\hat{m}) \\
& &\,+ \sum_{i=1}^{M-1} \left(\frac{1}{s_{i+1}}- \frac{1}{s_{i}}\right)\PP(\hat{m}-M_{\{1:i\}})
\,+\,\sum_{i=M+1}^{n} \left(\frac{1}{s_{i-1}}- \frac{1}{s_{i}}\right)\PP(\hat{m}-M_{\{i:n\}}), \nonumber 
\end{eqnarray}
where $M_{\{1:i\}}$ and $M_{\{i:n\}}$ are defined in \eqref{MA}.
\end{thm}
\begin{proof}
Using  $\PP(a-b)=\PP(a)+\PP(b)-2\PP(a,b)$, we see that \eqref{mainVar} is equivalent to 
 \begin{align} \label{intermediaireVar}
&{V(m)} \nonumber\\
&= \frac{1}{s_M}\PP(\hat{m}) \,+\,\sum_{i=1}^{M-1} \left(\frac{1}{s_{i+1}}- \frac{1}{s_{i}}\right)\PP(M_{\{1:i\}})
\,+\,\sum_{i=M+1}^{n} \left(\frac{1}{s_{i-1}}- \frac{1}{s_{i}}\right)\PP(M_{\{i:n\}})\nonumber\\
&\, - \;2\left(\sum_{i=1}^{M-1} \left(\frac{1}{s_{i+1}}- \frac{1}{s_{i}}\right)\PP(\hat{m},M_{\{1:i\}}) \,+\,\sum_{i=M+1}^{n} \left(\frac{1}{s_{i-1}}- \frac{1}{s_{i}}\right)\PP(\hat{m},M_{\{i:n\}})
\right). 
\end{align}
We show  that the right hand sides of \eqref{varcomplic} and \eqref{intermediaireVar} are the same.
Below, we expand \eqref{varcomplic} using  $\PP(a+b)=\PP(a)+\PP(b)+2\PP(a,b)$ and compute the coefficients in the expanded formula.
Note that for all $Z\in Z_G$, $\PP(M_{\{1:i\}},M_{\{k:n\}})Z=0$ for all $i\leq M-1$ and $k\geq M+1$, since ${Z_{\{1:i\},\{k:n\}}}=0$. 

For a fixed $r\leq M-1$, the coefficient of $\PP({M_{\{1:r\}}})$ is
\[ \frac{s_r-s_{r+1}}{s_r^2}+ \sum_{i=r+1}^{M-1}  (s_i-s_{i+1}) \left(\frac{1}{s_r}-\frac{1}{s_{r+1}} \right)^2
+s_M \left(\frac{1}{s_r}-\frac{1}{s_{r+1}} \right)^2
 = \frac{1}{s_{r+1}}-\frac{1}{s_r}.
\]
By a mirror argument, for a fixed $r\geq M+1$, the coefficient of $\PP({M_{\{r:n\}}})$ is
$ \frac{1}{s_{r-1}}-\frac{1}{s_r}$.
{On the other hand,} 
 the coefficient of $\PP(\hat{m})$ is $\frac{1}{s_M}$.

{For a fixed $r$,} the coefficient of $\PP(\hat{m}, M_{\{1:r\}} )$ is $\frac{1}{s_r}-\frac{1}{s_{r+1}}$ {if $r \leq M-1$},
{and} the coefficient of $\PP(\hat{m}, M_{\{r:n\}} )$ is $\frac{1}{s_r}-\frac{1}{s_{r-1}}$ {if $r \geq M+1$}.
{Moreover, if} $k < r\leq M-1$, the coefficient of $\PP(M_{\{1:r\}},\,M_{\{1:k\}})$ is
\begin{align*} 
 & (s_r-s_{r+1})\frac{1}{s_r}  \left(\frac{1}{s_k}-\frac{1}{s_{k+1}} \right)
  +  \sum_{i=r+1}^{M-1}  (s_i-s_{i+1}) \left(\frac{1}{s_r}-\frac{1}{s_{r+1}} \right) \left(\frac{1}{s_k}-\frac{1}{s_{k+1}} \right)\\
 &\, +\,  s_M \left(\frac{1}{s_r}-\frac{1}{s_{r+1}} \right) \left(\frac{1}{s_k}-\frac{1}{s_{k+1}} \right)\\
 &= \left(\frac{1}{s_k}-\frac{1}{s_{k+1}}\right)\left(1-\frac{s_{r+1}}{s_r}+s_{r+1}\left(\frac{1}{s_r}-\frac{1}{s_{r+1}} \right) \right) = 0.
\end{align*}
By a mirror argument, for a fixed $M+1\leq k < r$, the coefficient of $\PP(M_{\{k:n\}},{M_{\{r:n\}}})$ is $0$.
\end{proof}

\begin{rem}
 $\hat{m}$ is easy to compute, using, for non adjacent $i$ and $j$ the formula\\
$
\hat{m}_{ij}=m_{i,V\setminus\{i,j\}} (\hat{m}^{-1}_{V\setminus\{i,j\},V\setminus\{i,j\}}) m_{V\setminus\{i,j\} ,j}.
$  \cite[p.1279]{L-M}.
\end{rem}

In the next Corollary, we consider $\us=p\mathbf{1}$, $p>1/2$. 
We note that 
$\delta_{p\mathbf{1}}^{(M)}$ and $\gamma_{p\mathbf{1},\,y}^{(M)}:=\gamma_{p,y}$ do not depend on $M$.
\begin{cor}
The variance function of the Wishart exponential family $\gamma_{p,y}$ is 
$$
V(m)=\frac1p \PP(\hat m).
$$
\end{cor}

\subsubsection{A relation between the inverse mean map  and $m_{\frac1\us}$ }
Recall that for the classical Wishart exponential families $W_{s{\bf 1}, y}$
on the symmetric cone $Sym^+_n$ the bijection between the cone $Q_G$ and $P_G$ is given by 
$L(m)=m^{-1}$. 
The mean map is $m_s(y)=sy^{-1}$ and the inverse mean map $\psi_s(m)=sm^{-1}$.
It follows that
$$
\psi_s=L\circ m_{\frac1s} \circ L,
$$
that is, the maps $\psi_s$ and $ m_{\frac1s}$ are intertwined by the bijection $L$.

An analogous property
holds on the cone $Q_{A_n}$, with the intertwiner given by the Lauritzen map. 
The bijection  $L: Q_{A_n}\rightarrow P_{A_n}$
is the Lauritzen map  $L(m)=(\hat m)^{-1}$. 
Its inverse  $L^{-1}: P_{A_n}\rightarrow Q_{A_n}$
is $L^{-1}(y)=\pi(y^{-1})$.

 \begin{pro}
The inverse mean map $\psi_{\us}^{(M)}: Q_G\rightarrow P_G$
satisfies 
$$
\psi_{\us}^{(M)}=L \circ m_{\frac1{\us}}^{(M)} \circ L.
$$
Equivalently, for any $m\in Q_G$, \,
$
\pi(\psi_{\us}^{(M)}(m)^{-1})= m_{\frac1\us}^{(M)}(\hat m^{-1}).
$

$$\xymatrix{
 Q_{A_n} \ar[r]^{\psi_{\us}^{(M)}} \ar[d]_{L}&  P_{A_n}\\
 P_{A_n}\ar[r]_{m_{\frac1\us}^{(M)}}&  Q_{A_n} \ar[u]_{L}
}$$

\end{pro}
\begin{proof}
Using  formula \eqref{msM} of the mean function and definition \eqref{MA} of $M_{\{1:i\}}$ and $M_{\{i:n\}}$,
 we {see that}
 $m_{\frac1\us}^{(M)}(\hat m^{-1})$ equals
\begin{align*}
  {\pi}\left(
  \sum_{j=1}^{M-1} \left(\frac1{s_j}-\frac1{s_{j+1}}\right){M}_{\{1:j\}} + \frac{ \hat{m}}{s_M}+
\sum_{j=M+1}^{n} \left(\frac1{s_j}-\frac1{s_{j-1}}\right)M_{\{j:n\}}\right). 
\end{align*}
Confronting this result with \eqref{y-1}, we obtain \,
$
m_{\frac1\us}^{(M)}(\hat m^{-1})= {\pi}\left( \psi_{\us}^{(M)}(m)^{-1}\right).
$

\end{proof}
\subsection{Quadratic construction of Riesz measures and  Wishart {distributions} on $Q_G$}\label{quadr_Q}

Let $I\subset\{1,\ldots,n\}$.
We define $|I|$-dimensional subspaces $W_I$ of $\R^n$ by
$$
W_I=\{ x\in \R^n|\ x_i=0,\ i\notin I\}.
$$
Denote by $q^I$ the quadratic map
$q^I(x)=  x\transp{x}$
 from $W_I$ into  
 Sym$(n,\R)$
 and by   $q^I_{\ast}$ its projection onto ${I_G}$, i.e.
 $q^I_{\ast}={\pi}\circ q^I$. 
The maps  $q^I_{\ast}$
 are clearly $Q_G$-positive (submatrices {$y_I$} of a positive definite matrix {$y$} are positive definite for any ${I}\subset \{1,\ldots,n\} ).$
In \cite{graczykIshi}, p.322,
 Riesz measures $\mu_q$  associated  to a quadratic map $q$ were defined and their Laplace transform computed. Recall that the  measure $\mu_{q^I_{\ast}}$ is the image  of the Lebesgue measure on $W_I$ by $q^I_{\ast}$
and that its Laplace transform equals
\begin{equation}\label{Lap_quadr}
\mathcal{L}(\mu_{q^I_{\ast}})(y)= \pi^{|I|/2} |y_I|^{-1/2},   \ \ y\in P_G.
\end{equation}

 When $I=\{1,\ldots,k\}$, we write  $q^I_{\ast}=q^k_\ast$.
 When  $I=\{k,\ldots,n\}$, we write  $q^I_{\ast}=\tilde q^k_\ast$.
 
 Fix $M\in\{1,\ldots, n\}$. We define the set $B_M$ of {\it  basic   quadratic maps}  for the Riesz  $R_{\us}^{(M)}$ and Wishart 
 $\gamma^{(M)}_{\us, y}$ families on $Q_G$  by
 $ B_M=\{q^1_\ast,\ldots, q^{M-1}_\ast, q^n_\ast,\tilde q^{M+1}_\ast,\ldots, \tilde q^{n}_\ast\}$.
 Note that the basic   quadratic maps with values in $Q_G$ are different 
 for each fixed $M=1,\ldots, n$.

\begin{pro}
Let $\sigma_i \in  \R, i=1,\ldots,m$.
A virtual quadratic map

$  \hspace{2cm}
q^{\usigma}_{\ast}= \sum_{i<M}^{\oplus} (q^i_\ast)^{\oplus \sigma_i}
\oplus (q^n_\ast)^{\oplus \sigma_{M}}\oplus 
 \sum_{i>M}^{\oplus} (\tilde q^{i}_\ast)^{\oplus \sigma_{i}}.
$\\
exists {if} there exists $\us$ satisfying  $s_i >\frac{1}{2}, i\neq M$, $s_M>0$
and
\begin{eqnarray}\label{sigma=s}
\frac{\sigma_i}2=s_i-s_{i+1}, 1\le i<M,\ \ \  
 \frac{\sigma_{M}}2=s_M,
\ \ \ 
  \frac{\sigma_{i}}2=s_i-s_{i-1}, M<i\le n.
\end{eqnarray}
\end{pro}
\begin{proof}
We compare  the Laplace transform of $\mu_{q^{\usigma}_{\ast}}$, computed thanks to \eqref{Lap_quadr}, with \eqref{LapRieszQ}.  
{As a result, we} see that there exists a constant $c>0$ such that
$R_{\us}^{(M)}= c \mu_{q^{\usigma}_{\ast}}$.
\end{proof}

Thus all the Riesz  $R_{\us}^{(M)}$  measures on $Q_G$  defined in this paper are obtained 
as virtual
or true  (i.e.
for $\sigma_i\in\N$) quadratic Riesz families, with basic 
maps from $B_M$.

Observe that by the quadratic construction, 
 we can obtain absolutely continuous  Riesz measures on $Q_G$ not belonging to $\cup_M \{R_{\us}^{(M)}\} $,
 e.g. when $n=3$, consider $\mu_q$ associated to the quadratic map $q=q_\ast^2\oplus (q_\ast^3)^{\oplus 2} \oplus \tilde q_\ast^2$.
 
\subsubsection{Relation to missing data}
Let us mention a nice application of the (true)
 quadratic Wishart {distributions} constructed from the basic set $B_M$, as laws of the Maximum Likelihood Estimators for the covariance
 matrix $\Sigma$ restricted to the terms $(\Sigma_{ij})_{i\sim j}$, i.e. to the three neighbour diagonals,
 in a two-sided monotonous missing data problem, when the sample contains:
   $\sigma_i$ observations of $(X_1,\ldots, X_i)$, $i<M$,\,
 $\sigma_k$ observations of $(X_k,\ldots, X_n)$, $k>M$ 
 and $\sigma_M$ observations of the complete $n$-dimensional
 Gaussian character  $(X_1,\ldots, X_n)$.
 
 \subsection{Higher order moments of Wishart families on $Q_{A_n}$ }\label{higher}
 Thanks to the identification of Wishart families $\gamma^{(M)}_{\us, y}$ with quadratically constructed Wishart {distributions} 
 $\gamma_{q^{\usigma}_\ast}$ in Section  \ref{quadr_Q}, we can
 compute  moments of any order $N$ of a Wishart random variable $X$
 on $Q_{A_n}$.
\begin{thm}	\label{highMom}
Let $X$ be a $Q_{A_n}$-valued random  variable with the Wishart
law  $\gamma^{(M)}_{\us, y}$.
Let $z^{(1)} ,z^{(2)}, \ldots z^{(N)}\in Z_G$.
Then, denoting by $C(\pi)$ the set of cycles of a permutation
 $\pi \in S_N$,
 {the $N$-th moment
 ${\bf E}(\langle X,z^{(1)}\rangle \ldots \langle X,z^{(N)}\rangle)$
 equals}
\begin{align*}
{}&\sum_{\pi\in S_N}
\prod_{{c\in C(\pi)}}
\left\{\sum_{i=1}^{M-1}(s_i-s_{i+1})\tr\prod_{j\in c}
(y_{\{1:i\}})^{-1}z_{\{1:i\}}^{(j)}\right. \\
&  \phantom{=\sum_{\pi\in S_N}\prod_{{c\in C(\pi)}}} \quad
\left. + s_M\tr \prod_{j\in c}y^{-1}z^{(j)}
+\sum_{i=M+1}^{n}(s_i-s_{i-1})\tr\prod_{j\in c}
(y_{\{i:n\}})^{-1}z_{\{i:n\}}^{(j)}
\right\} .
\end{align*}
\end{thm}
\begin{proof}
We apply Theorem 2.13   from \cite{graczykIshi}  and formula \eqref{sigma=s}.
\end{proof}
 
\begin{cor}
If $\us=s{\bf 1}$, $s>\frac12$, then $\gamma^{(M)}_{\us, y}=\gamma_{s, y}$ does not depend on $M$.
{Moreover,} for $X$ with law $\gamma_{s, y}$, {we have}
  $${\bf E}(\langle X,z^{(1)}\rangle
\ldots \langle X,z^{(N)}\rangle)=
\sum_{\pi\in S_N} s^{|C(\pi)|}
\prod_{{c\in C(\pi)}}\tr \prod_{j\in c}y^{-1}z^{(j)} .$$
\end{cor}
{\bf Example.} 
 For any graph $A_n$ and $N=3$ we get for $X$ with law $\gamma_{s, y}$:\\
 $
 {\bf E}(\langle X,z^{(1)}\rangle
\langle X,z^{(2)}\rangle \langle X,z^{(3)}\rangle)=
s^3\prod_{j=1}^3\tr(y^{-1}z^{(j)}) +\\
s^2[\tr y^{-1}z^{(1)}y^{-1}z^{(2)} \tr y^{-1}z^{(3)}
+ \tr y^{-1}z^{(1)}y^{-1}z^{(3)} \tr y^{-1}z^{(2)}+\\ 
\tr y ^{-1}z^{(2)}y^{-1}z^{(3)}\tr y^{-1}z^{(1)}
]
+s[\tr\prod_{j=1}^3y^{-1}z^{(j)} +  \tr y^{-1}z^{(1)}y^{-1}z^{(3)}y^{-1}z^{(2)}].
$

\section{ Wishart exponential families   on the cone $P_G$. }\label{WishartP}

 A measure $\tilde{R}$ on $P_G$ is said to be a Riesz measure if, for some $1\leq M \leq n$,  $s_M>-1$ and $s_i>-3/2$, $i\neq M$, \,
 its Laplace transform is given by 
 \begin{equation}
  L_{\tilde{R}}(x)= \int_{P_G} e^{-\langle x ,y \rangle}\tilde{R}(dy) = \delta_{-\us}^{(M)}(x)\varphi_{Q_{A_n}}(x).
 \end{equation}

 From formula \eqref{eqn:laplacedelta}, the measure $\tilde{R}_{\us}^{(M)}(dy) = C_{\us} \Delta_{\us}^{(M)}(y)dy $, where 
 $$C_{\us}^{-1}=\pi^{(n-1)/2}
  \Bigl\{\prod_{i\neq M}\Gamma(s_i+\frac{3}{2})\Bigr\}
 \Gamma(s_M+1),$$ is a Riesz measure.
 The exponential family of generated by $\tilde{R}_{\us}^{(M)}$ will be called the exponential family of Wishart distributions on 
 $P_{G}$. Its density function is 
 \begin{equation}
 \tilde{\gamma}_{\us}^{(M)}(y)= \frac{1}{\delta_{-\us}^{(M)}(x)\varphi_{Q_{A_n}}(x)} e^{-\langle x ,y \rangle}\tilde R_{\us}^{(M)} (y). 
 \end{equation}
Its Laplace transform is 
\begin{equation}
 L_{\tilde{\gamma}_{\us}^{(M)}}(\theta) = \int_{P_G} e^{-\langle \theta ,y \rangle} \tilde{\gamma}_{\us}^{(M)}(y)
 = \frac{L_{\mu_{\us}^{(M)}}(\theta +x)}{L_{\mu_{\us}^{(M)}}(x)}
 =\frac{\delta_{-\us}^{(M)}(\theta +x)\varphi_{Q_{A_n}}(\theta+x)}{\delta_{-\us}^{(M)}(x)\varphi_{Q_{A_n}}(x)}.
\end{equation}

\subsection{Mean and covariance}
\begin{thm}
 The mean function of the Wishart exponential family on $P_G$ is for all $s_i > -\frac{3}{2}$ and $x\in Q_G$,
 \begin{align}\label{meanP}
 \tilde{m}_{\us}^{(M)}(x)&= \sum_{i=1}^{M-1}(s_i+\tfrac{3}{2}) (x_{\{i : i+1\}}^{-1} )^0
 + \sum_{i=M+1}^n (s_i+\tfrac{3}{2}) (x_{\{i-1 : i\}}^{-1} )^0 
 \\
 &   -\sum_{i=2}^{M-1}(s_{i-1}+1) (x_{ii}^{-1})^0 
 -(s_{M-1}-s_M+s_{M+1}+1) ( x_{MM}^{-1})^{0} \nonumber\\
& -  \sum_{i=M+1}^{n-1}(s_{i+1}+1) ( x_{ii}^{-1})^{0}. \nonumber
 \end{align}
The covariance function {$\tilde v(x) : I_G \to Z_G$} of the Wishart exponential family on $P_G$  equals
 \begin{align*}
 \tilde v(x)&= \sum_{i=1}^{M-1}(s_i+\tfrac{3}{2}) {\PP}\left[(x_{\{i : i+1\}}^{-1} )^0\right]
 + \sum_{i=M+1}^n (s_i+\tfrac{3}{2}) {\PP}\left[(x_{\{i-1 : i\}}^{-1} )^0 \right]
 \\
 &   -\sum_{i=2}^{M-1}(s_{i-1}+1){\PP}\left[ (x_{ii}^{-1})^0 \right]
 -(s_{M-1}-s_M+s_{M+1}+1){\PP}\left[ ( x_{MM}^{-1})^{0}\right] \nonumber\\
& -  \sum_{i=M+1}^{n-1}(s_{i+1}+1){\PP}\left[ ( x_{ii}^{-1})^{0}\right], \nonumber
 \end{align*}
{where we identify $I_G$ with $Z_G$ by the trace inner product.}
\end{thm}
\begin{proof}
We have
$
 \tilde{m}_{\us}^{(M)}(x) = -\grad \log L_{\mu_{\us}^{(M)}}(x) =  -\grad \log \delta_{-\us}^{(M)}(x)\varphi_{Q_{A_n}}(x).
$
The covariance operator is obtained by differentiation of  \eqref{meanP}.
 \end{proof}
 \subsection{Quadratic construction of Riesz measures and  Wishart {distributions} on $P_G$}\label{quadr_P}

Let 
$M\in\{1,\ldots n\}$. Suppose  $s_i > -\frac{3}{2}$, for all $i\neq M$ and $s_M>-1$. 
Let $\theta\in Q_G$.
In order to establish a relation between quadratically constructed
Riesz measures $ \tilde \mu_q$ on $P_G$ and the measures $\tilde R_{\us}^{(M)}$
we consider the sets 
$J_k=\{k,k+1\}$ and $J_k'=\{k\}$.
As {\it basic quadratic maps} we choose the quadratic $P_G$-positive maps $q^{J_k }$ and  $q^{J_k'}$.
For $\alpha=(\alpha_1,\ldots,\alpha_{n-1})$ and $\beta=(\beta_1,\ldots,\beta_n)$ with $\alpha_i,\beta_j\in\N$ define
$
q^{\alpha,\beta}=\sum^{\oplus}_{k<n} (q^{J_k })^{\oplus \alpha_k}
\oplus \sum^{\oplus}_{k\le n} (q^{J_k' })^{\oplus \beta_k}$.
The following proposition is easy to prove by comparing  
$\delta^{(M)}_{-\us}(\eta)\varphi_{Q_{A_n}}(\eta)$ with  the Laplace transform
 \begin{align*}
{L}_{\tilde \mu_{q^{\alpha,\beta
}}}(\eta)=
\pi^{({\sum_{k<n} \alpha_k+\sum_{k\le n} \beta_k})/2}
\prod_{i<n}|\eta_{\{i,i+1\}}|^{-\alpha_i/2}
\prod_{j\le n}|\eta_{jj}|^{-\beta_j/2}.
\end{align*}
 \begin{pro}\label{Quadr_PAn}
 Let $M\in\{2,\ldots,n-1\}$. Then  there exists a constant $c>0$ such that $c\tilde R_{\us}^{(M)}=\tilde \mu_{q^{\alpha,\beta}}$ if and only if  $\alpha_i/2=s_i+3/2$, $i\le M-1$, $\alpha_i/2=s_{i+1}+3/2$, 
 $i\ge  M$, $\beta_1=0$,  $\beta_i/2=-s_{i-1}-1,2\le i\le M-1, \beta_M/2=-s_{M-1}+s_M-s_{M+1}-1, \beta_i/2=-s_{i+1}-1,M+1 \le i<n, \beta_n=0$.  For $M=1,n$ the condition $\beta_M=0$ is supressed.
 \end{pro}
Proposition \ref{Quadr_PAn} implies easily two following facts. 
\begin{cor}\label{analysis}
\begin{enumerate}
\item All Riesz measures $\tilde R_{\us}^{(M)}$ are equal (up to a factor)
	to a virtual  quadratic Riesz measure  $\tilde \mu_{q^{\alpha,\beta}}$.
\item For $n\ge 4$, 
no true quadratic Riesz measure  $\tilde \mu_{q^{\alpha,\beta}}$, $\alpha_i,\beta_j\in\N$,
	 belongs (up to a factor) to the set of Riesz measures $\tilde R_{\us}^{(M)}$.
\end{enumerate}
\end{cor}
\begin{proof} To prove Part 2,
we have  conditions $s_i+\frac32=\alpha_{i'}/2$
and $s_j+1=-\beta_{j'}/2$, so all (except at most one) $s_i\ge -1$,   and  all (except at most one) $s_j\le -3/2$ simultanously.
\end{proof}	
	
	\subsection{Higher order moments of Wishart families on $P_{A_n}$}
	Thanks to Part 1 of Corollary \ref{analysis}, all the moments of the Wishart Exponential Families    $\tilde\gamma_{\us,\theta}^{(M)}$ can be computed, using  Theorem 2.13  
	from \cite{graczykIshi}
	and Proposition \ref{Quadr_PAn}.
	 
	\begin{thm}	
Let $Y$ be a $P_{A_n}$-valued random  variable with the Wishart
law  $\tilde\gamma^{(M)}_{\us, \theta}$.
Let $x^{(1)} ,x^{(2)}, \ldots x^{(N)}\in I_G$.
Then, denoting by $C(\pi)$ the set of cycles of a permutation $\pi \in S_N$,
 {the $N$-th moment 
${\bf E}(\langle Y,x^{(1)}\rangle \ldots \langle Y,x^{(N)}\rangle)$
 equals}
\begin{align*}
{}&\sum_{\pi\in S_N}
\prod_{{c\in C(\pi)}} 
\Bigl\{
\sum_{i=1}^{M-1}(s_i+\frac32)\tr\prod_{j\in c}
(\theta_{\{i:i+1\}})^{-1}x_{\{i:i+1\}}^{(j)} +\\
&\sum_{i=M}^{n-1}(s_{i+1}+\frac32)\tr\prod_{j\in c}
(\theta_{\{i:i+1\}})^{-1}x_{\{i:i+1\}}^{(j)}
-\sum_{i=2}^{M-1}
(s_{i-1}+1)\theta_{ii}^{-|c|}\prod_{j\in c} x_{ii}^{(j)}\\
&-(s_{M-1}-s_M+s_{M+1}+1)\theta_{MM}^{-|c|}\prod_{j\in c} x_{MM}^{(j)}
- \sum_{i=M+1}^{n-1}(s_{i+1}+1) \theta_{ii}^{-|c|}\prod_{j\in c} x_{ii}^{(j)}
\Bigr\} .
\end{align*}
\end{thm}	

\section{Relations with the type {I} and type {II} Wishart distributions of \cite{L-M}}
\label{LM}
In this section we will explain the relation between our work and 
  type 1 and type 2 Wishart distributions constructed by \cite{L-M}.
  
  \cite{L-M} introduced, studied  and used  the function $H(\alpha,\beta,x)$ on $ Q_G$ as a
  generalized  power function for constructing  type {I} and type {II} Wishart 
distributions. The reader is referred to the {cited} paper
 for the general definition of the  function  $H(\alpha,\beta,x)$  as well as
for graphical theoretic notions such as cliques, separators and perfect order of cliques
(see also 
\cite{lauritzen1996}). For our purpose, it is sufficient to recall that for  $ \alpha\in\R^{n-1}$
and 
  $\beta \in\R^{n-2}$ 
\begin{equation}\label{Hformula}
H(\alpha,\beta;x)=\frac{\prod_{i=1}^{n-1}|x_{\{i,i+1\}}|^{\alpha_i}}{\prod_{i=2}^{n-1}x_{ii}^{\beta_i}},\ \ x\in Q_{A_n}, 
\end{equation}
that the cliques(i.e.  the sets of vertices of maximal complete subgraphs) are  $\{1,2\},\ldots,\\
\{n-1,n\}$ and the separators $\{2\},\ldots,
\{n-1\}$.
The definition of the function $H(\alpha,\beta;x)$ does not include any restrictions on the values of the  parameter 
$(\alpha,\beta)$ of dimension $2n-3$. 

However, the existence of type {I} Wishart distributions
on $Q_G$ is only showed for $(\alpha,\beta)$ belonging
to some set $A_P$ dependent on a perfect order of cliques $P$,
i.e. for 
$(\alpha,\beta)\in {\mathcal A}_0=
 \cup_P A_P$,
   where the union is on all   perfect order of cliques.
Proposition \ref{A_P} describes this set for $A_n$ 
graphs. 
It also makes clear a phenomenon observed by \cite{L-M} for the graph $A_4$, where there are only two different sets 
$A_P$ although  there are $4$ perfect orders of cliques.
To prove Proposition  \ref{A_P} we use the following  explicit  relation between two concepts: 
perfect orders of cliques used by \cite{L-M} and eliminating orders of vertices used in this work. 

 \begin{pro}\label{POC}
Let  $G=A_n: 1-2-3-\hdots-n$. A clique ordering $C'_1< \hdots < C'_{n-1}$ is perfect if and only if $C'_{n-1}\prec  \hdots \prec  C'_1$ is
an eliminating order on the $A_{n-1}$  graph $G': C_1-C_2 \hdots -C_{n-1}$. There are  $2^{n-2}$ perfect orders of cliques
on $A_n$.
\end{pro}
\begin{proof}
The proof is in two parts, for the two inclusions of the claimed equality.  Both parts are straightforward and
   based on  the definitions of a perfect order of cliques and of an eliminating order on a graph.
   We omit the details.
\end{proof}

\begin{pro}\label{A_P}
  Let $P':  C_1'<C'_2<\ldots < C'_{n-1}$ and $P'': C_1''<C''_2<\ldots < C''_{n-1}$ be two perfect orders of cliques on $G=A_n$. Let $S_2'$ and $S_2''$ be the first separators of $P'$ and $P''$.
  If $S_2'=S_2''$ then $A_{P'}=A_{P''}$, i.e. the parameter set $A_P$ depends only on the first separator $S_2$ with respect to the clique order $P$.
  If $S_2=\{M\}$ then the set $A_P$ is described by the conditions:\\
  \begin{equation}\label{A1}
  \begin{cases}
    \alpha_{j}=\beta_{j+1}\ {\it if}\ 1\le j\le M-2,\\
 \alpha_{j}=\beta_{j}\ {\it if}\ M+1\le j\le n-1,
  \end{cases}\end{equation}
  {and}
\begin{equation}\label{A2}
  \alpha_{j}>\frac12  \; \textnormal{for all}\; 1\leq j\leq n-1;\;  \alpha_{M-1} + \alpha_{M} -\beta_M >0.
  \end{equation}
  Thus  $\mathcal{A}_0= \cup_P A_P$ is the set of $(\alpha,\beta)$ such that there exists $2\leq M\leq n-1$ {for which} 
  \eqref{A1} and \eqref{A2} are satisfied. 
 \end{pro}
\begin{proof}
 We use Propositions \ref{order} and  \ref{POC}. 
\end{proof}

The reference measure $\mu_G$ used by \cite{L-M}  is, on the cone  $Q_{A_n}$,
  \begin{equation}\label{mu}
 \mu_{A_n}(x)(dx)= H_{A_n}(-\frac{3}{2}\mathbbm{1},-\mathbbm{1};x)1_{Q_{A_n}}(x)dx. 
\end{equation}
By \eqref{char},  we observe that $\mu_{A_n}(x)(dx)= \varphi_{Q_{A_n}}(x)1_{Q_{A_n}}(x)dx$.
{Namely}, the reference measure $\mu_G$ is the characteristic measure of the cone $G=Q_{A_n}$.

\begin{thm}\label{thm3.3}(\cite{L-M} Theorem 3.3)
If $(\alpha,\beta)\in \mathcal{A}_0$, then, for a constant ${\Gamma_1}_{(\alpha,\beta)}$, and
for all $y\in P_{A_n}$
 \begin{equation*}
  \int_{Q_{A_n}}e^{-\tr(xy)} H(\alpha,\beta;x)\mu_{A_n}(x)(dx)=  {\Gamma_1}_{(\alpha,\beta)}H(\alpha,\beta; {\pi(y^{-1})}).
 \end{equation*}
\end{thm}
The methods of our article give a new simple proof of Theorem \ref{thm3.3}, see the proof
of Corollary \ref{comparison} below.\\

Let us  compare now the functions  $H(\alpha,\beta;x)$ and $H(\alpha,\beta; {\pi(y^{-1})})$
with the generalized power functions $\delta_{\us}^{(M)}$ and $\Delta_{\us}^{(M)}$.
 \begin{pro}\label{comp-H-delta}
\begin{enumerate}
\item Let $\alpha\in \R^{n-1}$ and $\beta\in \R^{n-2}$. There exists $\us\in\R^n$ such that
 $H(\alpha,\beta; x)=\delta_{\us}^{(M)}(x)$
  if and only if 
   \eqref{A1} holds for some $2\le M\le  n-1$.
  
    Then  $s_j= \alpha_{j}\ {\it if}\ 1\le j\le M-1,\ 
   s_M=\alpha_{M-1} + \alpha_{M} -\beta_M$ and
$ s_j=\alpha_{j-1}\ {\it if}\ M+1\le j\le n.$
\item Moreover, under the hypothesis of {Part 1},  we have $H(\alpha,\beta;\pi(y^{-1}))=\Delta_{-\us}^{(M)}(y)$. 
\end{enumerate} 
\end{pro}
 \begin{proof}
The equality of $H(\alpha,\beta; x)$ and $ \delta_{\us}^{(M)}(x)$ is easily verified  by confronting their definitions  \eqref{Hformula} and \eqref{delta(M)}.  
Part 2 follows from Theorem \ref{delta-Delta}.
 \end{proof}
 
 \begin{cor}\label{comparison}
 The type {I} Wishart distributions indexed by the set ${\mathcal A}_0$ are equal   to  the  subset $\bigcup_{M=2}^{n-1} (\gamma_{\us,y}^{(M)})_{y\in P_G}$ of Wishart NEF families defined in Section \ref{RieszWishartQ}.  
Thus  they are strictly contained
  in the set of all   Wishart NEF families on $Q_G$, equal to  $\bigcup_{M=1}^{n} (\gamma_{\us,y}^{(M)})_{y\in P_G}$. 
 \end{cor}
 \begin{proof}
 It is a direct application of Proposition \ref{comp-H-delta} and Theorem \ref{laplacedelta}.
 Note that  Theorem \ref{laplacedelta} implies  Theorem {\ref{thm3.3}} of  \cite{L-M}.
 
  The family of  functions $H(\alpha,\beta,x)$ {does not contain} the power functions $\delta_{\us}^{(1)}$ or $\delta_{\us}^{(n)}$.
 In fact,  the last functions
  contain   powers of $n-1$ {diagonal} elements $x_{ii}$, whereas
  the function $H(\alpha,\beta,x)$ contains powers of
  $n-2$ such elements.
 \end{proof}
 
   Similar comparisons can be done on the cones $P_G$. In this case, \cite{L-M}
   define  type {II} Wishart distributions on $P_G$ indexed by a set ${\mathcal B}_0$, analogous
   to the set ${\mathcal A}_0$ for $Q_G$. Similar
    arguments as on the cone $Q_G$ lead to
 \begin{cor}\label{comparisonP}
 The type {II} Wishart distributions on $P_G$ indexed by the set ${\mathcal B}_0$
  are equal   to  the  subset $\bigcup_{M=2}^{n-1} (\tilde\gamma_{\us,x}^{(M)})_{x\in Q_G}$ of Wishart NEF families defined in Section \ref{WishartP}.  
Thus  they are strictly contained
  in the set of all   Wishart NEF families on $P_G$, equal to  $\bigcup_{M=1}^{n} (\tilde\gamma_{\us,x}^{(M)})_{x\in Q_G}$. 
 \end{cor}  

  \section{Appendix}
We list here some properties of triangular matrices, used in proofs.

\begin{lem}\label{a:triang}
\begin{enumerate}
\item Let $ A= K^0$, where $K=A_{\{1:k\}}$
and let $L$ be lower triangular and $U$ upper triangular  $n\times n$ matrices.
Then
$UAL=\left(
U_{\{1:k\}}KL_{\{1:k\}}\right)^0.
$
 \item Let $M,L,U$ be  matrices $n\times n$, with $L$ lower triangular and $U$ upper triangular. Then, for all $i=1,\ldots, n$,
$(LMU)_{\{1:i\}}= L_{\{1:i\}}M_{\{1:i\}} U_{\{1:i\}}$ and $(UML)_{\{i:n\}}=U_{\{i:n\}}M_{\{i:n\}}L_{\{i:n\}}$.

\item If $T$ is an  invertible triangular matrix then
 $(T_{\{1:k\}})^{-1}=(T^{-1})_{\{1:k\}}$ for all $k=1,\ldots, n$.
 \end{enumerate}
 
\end{lem}
 All these properties are elementary 
and easy to prove,  by block multiplication of matrices (1,2) or by inverse matrix formula with cofactors (3).





\end{document}